\documentclass{l4dc2024}

\title[Multi-Agent Coverage Control with Transient Behavior Consideration]{Multi-Agent Coverage Control with Transient Behavior Consideration}
\usepackage{times}
\usepackage[utf8]{inputenc}
\usepackage{natbib}
\usepackage{graphicx}
 \usepackage[top=2cm, bottom=2cm, left=2.3cm, right=2.3cm]{geometry}
 \usepackage{algorithm}
  \usepackage{algorithmic}
 \usepackage{amsfonts}
 \usepackage{amsmath}

 \usepackage{amssymb}
 \usepackage{float}
 \usepackage{caption}
 \usepackage{multirow}
 \usepackage{wrapfig}
 \DeclareMathOperator*{\argmax}{argmax}
 
 \usepackage{cite} 
\newenvironment{talign}
 {\align}
 {\endalign}
\newenvironment{talign*}
{\csname align*\endcsname}
 {\endalign}

\newtheorem{rmk}{Remark}
\newtheorem{defi}{Definition}

\newcommand{\bE}{\mathbb{E}}
\newcommand{\GP}{\mathcal{GP}}
\newcommand{\UCB}{\textup{UCB}}
\newcommand{\Oracle}{\texttt{Oracle}}
\newcommand{\lmax}{\lambda_{\max}}
\newcommand{\tr}{\textup{trace}}
\newcommand{\wk}{\widetilde{\kappa}}
\newcommand{\wK}{\widetilde{K}}
\newcommand{\wX}{\widetilde{X}}
\newcommand{\wY}{\widetilde{Y}}
\newcommand{\wDelta}{\widetilde{\Delta}}
\newcommand{\wLambda}{\widetilde{\Lambda}}
\newcommand{\blkdiag}{\textup{blkdiag}}
\newcommand{\K}{K}
\newcommand{\eval}{{\textup{eval}}}
\newcommand{\bR}{\mathbb{R}}

\usepackage[normalem]{ulem}
\usepackage{xcolor}
\allowdisplaybreaks

\author{%
 \Name{Runyu Zhang}\thanks{ The first two authors contributed equally to this work.}  \Email{runyuzhang@fas.harvard.edu}
 \AND
 \Name{Haitong Ma}\footnotemark[1]  \Email{haitongma@g.harvard.edu}
 \AND
 \Name{Na Li} \Email{nali@seas.harvard.edu}\\
 \addr School of Engineering and Applied Science, Harvard University%
 \thanks{This work was funded by NSF AI institute: 2112085, NSF CNS: 2003111, NSF ECCS: 2328241.}
}

\begin{document}

\maketitle
\vspace{-15pt}
\begin{abstract}%
This paper studies the multi-agent coverage control (MAC) problem where agents must dynamically learn an unknown density function while performing coverage tasks. Unlike many current theoretical frameworks that concentrate solely on the regret occurring at specific targeted sensory locations, our approach additionally considers the regret caused by transient behavior – the path from one location and another. We propose the multi-agent coverage control with the doubling trick (MAC-DT) algorithm and demonstrate that it achieves (approximated) regret of $\widetilde O(\sqrt{T})$ even when accounting for the transient behavior. Our result is also supported by numerical experiments, showcasing that the proposed algorithm manages to match or even outperform the baseline algorithms in simulation environments. We also show how our algorithm can be modified to handle safety constraints and further implement the algorithm on a real-robotic testbed. 
 \end{abstract}

\begin{keywords}%
  Multi-agent coverage control, Gaussian Processes, Bayesian optimization, no-regret learning %
\end{keywords}


\section{Introduction}
In multi-agent coverage control (MAC) problems, there is a group of agents collectively tasked with efficiently exploring and covering an environment typically characterized by certain density functions.
MAC problems find a lot of applications such as sensor networks \citep{Krause06}, search and rescue \citep{Wasim}, underwater exploration \citep{karapetyan2018multi} and habitat monitoring \citep{mainwaring2002wireless} etc. In the extensive studies of the MAC problem, classical approaches to coverage control \citep{cortes2004coverage,cortes2005coordination,cortes2005spatially,lekien2009nonuniform,hussein2007effective,bullo2012gossip,durham2011discrete} generally assume
a priori knowledge of the density function and employ Lloyd’s algorithm \citep{lloyd1982least} to guarantee the convergence of agents to a local minimum of the coverage cost. 

Recent investigations have expanded this paradigm to accommodate scenarios where the density function is unknown. In such cases, agents must simultaneously conduct coverage tasks and learn the density function dynamically. A prevalent approach involves modeling the density function as a Gaussian process (GP) and employing non-parametric learning through sensor measurements. In order to achieve good performance, agents need to balance exploration and exploitation, collecting informative samples to learn the density function (exploration) while concurrently converging to optimal coverage locations (exploitation). In particular, a specific subset of existing research focuses on the development of practical, adaptive, and distributed algorithms for coverage control utilizing the GP model \citep{luo2018adaptive, luo2019distributed, choi2008swarm, kemna2017multi, nakamura2022decentralized}. Despite their practical relevance, these endeavors lack a rigorous theoretical analysis for performance guarantees. An alternative strand of research, represented by \citep{carron2015multi, todescato2017multi}, offers asymptotic guarantees of convergence to near-local optimal solutions, yet the convergence rate is not studied. Recent contributions \citep{prajapat2022near, benevento2020multi, santos2021multi} have sought to address this void by proposing algorithms with convergence rate guarantees, particularly in the context of regret. It is noteworthy, however, that most papers only consider regret accumulated at the targeted sensory locations that agents travel between while neglecting regrets along the path where agents travel from one location to the next location. 
For applications that involve a physical moving sensor that can not move to a new point instantaneously or quickly, regret along the paths may be substantial. This observation motivates us to design an algorithm for the MAC problem with rigorous regret guarantees that explicitly account for the transient behavior. 



\paragraph{Our contribution:} 
In this paper, we study the MAC problem under an unknown density function. We present a novel algorithm, multi-agent coverage control with doubling trick (MAC-DT), demonstrating an $\widetilde O(\sqrt{T})$ regret, even when accounting for transient behavior. {The algorithm leverages the Upper Confidence Bound (UCB) technique, wherein, during each episode, it computes the UCB of the reward map and assigns a specific sensory location to each agent. Subsequently, each agent plans its path toward the designated location. The termination condition for each episode is determined by the `doubling trick' (detailed explanation in Algorithm \ref{alg:main})}. Our research is closely aligned with the works of \citep{prajapat2022near} and \citep{wei2021multi}. Our proposed algorithm bears resemblance to the MACOPT algorithm introduced by \citep{prajapat2022near}, employing the UCB of the Gaussian Process to navigate the trade-off between exploration and exploitation. However, \citet{prajapat2022near} do not consider transient behavior, while our algorithm incorporates this aspect by carefully designing episodes using the doubling trick. Though \citet{wei2021multi} addresses regret in the presence of transient behavior, due to differences in problem settings and algorithmic design, they achieve a slightly inferior regret rate of $\widetilde O(T^{2/3})$. Further, the regret defined in \citep{wei2021multi} is with respect to a local-optimal solution, where our (approximated-) regret is defined by the global-optimal solution. Our results are also supported by numerical studies, suggesting that the MAC-DT algorithm can match or even outperform the aforementioned baseline algorithms. Further, our algorithm can naturally be combined with safety considerations and operate in settings with obstacles or safety constraints. Lastly, we also validate the algorithm on a physical robotic testbed with three quadrotors covering an area. 

Due to space limit, we defer some of the auxiliary proofs and numerical details into the online version of the paper \citep{supp}.

\vspace{-5pt}
\paragraph{Other related works} In the setting where the density function is known, beyond Lloyd's algorithm-based approaches, there are alternative methods leveraging submodularity \citep{krause2011submodularity, nemhauser1978analysis,feige1998threshold} to address the MAC problem \citep{ramaswamy2016sensor,sun2017submodularity}. In scenarios where the density function is unknown, various papers \citep{schwager2009decentralized, schwager2015robust} explore parametric estimation as an alternative to GP modeling. These algorithms model the function as a linear combination of basis functions, aiming to learn the weights associated with each basis function. 
Furthermore, alternative strategies \citep{davison2014nonuniform, choi2010learning} take a distinctive route by not involving the identification of the unknown probability density function. Instead, they solely rely on random samples from the environment to determine the agents' coverage locations. Apart from traditional algorithmic approaches to coverage control, recent studies have also delved into the application of reinforcement learning techniques \citep{faryadi2021reinforcement, DIN2022deep, battocletti2021rl}, 
which incorporates the power of learning and adaptation into the field.

\vspace{-15pt}
\section{Problem Setup and Preliminaries } 
\vspace{-5pt}
\subsection{Multi-agent Coverage Control}
 \vspace{-5pt}

Consider the MAC problem on a connected directed graph $G = \{V,E\}$, where $V$ are the vertices of the graph and $E$ are the edges. The density function/reward map is given by $w: V\to \mathbb{R}^+$, where each vertex $v\in V$ is associated with a reward $w(v)\ge 0$. For notational simplicity we also use $w(S)$ to denote the entrywise function evaluation on the grid subset $S \subseteq V$. We use the notation $E(v) \subseteq E$ to denote the set of edges whose source vertex is $v$. It is also assumed that the graph is connected and with diameter $D$.  There are $N$ agents/robots located on the vertices of the graph and at each time step $t$ agent $i$ chooses an action $a_{i,t}$ from the edge set $ E(v_{i,t})$, where $v_{i,t}$ is agent $i$'s location at time $t$, and then transit to the end vertex of $a_{i,t}$. We assume that when located at vertex $v$, agent $i$ can cover a certain surrounding area of its current location, which is denoted as $s_i(v) \subseteq V$. For example, $s_i(v)$ can be just the vertex $v$ that agent $i$ is at, or a $\kappa$-hop neighborhood area of $v$. Also note that agents may have different covering ability, that is, two different agent $i,j$ may have different $s_i(v)$ and $s_j(v)$ even at the same location. For simplicity, we denote $s_{i,t} := s_i(v_{i,t})$. We also use $s_t := \cup_{i=1}^N s_{i,t}$ to denote the total area covered by all agents. It is assumed that at every location $v$ each agent's covering area is smaller than size $\K$, i.e. $|s_i(v)|\le \K$ for all $v\in V$. At each time step, each agent could select one vertex $v_{i,t}^\eval\in s_{i,t}$ and observe a noisy reward value which is a random variable $Y(v_{i,t}^\eval):= w(v_{i,t}^\eval) + \epsilon$, where $\epsilon\sim\mathcal{N}(0,\sigma^2)$ is some Gaussian noise. We also denote $s_t^\eval := \cup_{i=1}^N v_{i,t}^\eval$. For vertex $v$, $n_v(t)$ denotes the number of times that vertex $v$ has been sampled (i.e. chosen as the evaluation point) until time $t$, i.e. $n_v(t) = \sum_{\tau=1}^t \mathbf{1}\{v\in s_{\tau}^\eval\}$.

At time $t$, the total covered reward is given by $\|w(s_t)\|_1:=\sum_{v\in s_t}w(v)$, and the objective is to cover as much reward as possible. The optimal coverage value is denoted as $s^\star \!\!:=\! \argmax_{s = \cup_{i\!=\!1}^N \!s_{i}\!(v_i), v_i\!\in \!V} \!\|w(s)\|_1$.  For an algorithm $\texttt{Alg}$ that plans the path $v_{i,t}(\texttt{Alg})$ of the agents, we define the regret of the algorithm as
\begin{align}\label{eq:regret}
\vspace{-10pt}
   \textstyle R(\texttt{Alg}, T, w):= T\|w(s^\star)\|_1 -  \sum_{t=1}^T \|w(s_t(\texttt{Alg}))\|_1.
\vspace{-10pt}
\end{align}
We also define the $\alpha$-approximated regret as
\begin{align}\label{eq:approximated-regret}
\vspace{-10pt}
  \textstyle  R^\alpha(\texttt{Alg}, T, w):= \alpha T\|w(s^\star)\|_1 -  \sum_{t=1}^T \|w(s_t(\texttt{Alg}))\|_1.
\vspace{-10pt}
\end{align}
Compared with the regret definition in \citep{prajapat2022near} which considers regret accumulated only at the targeted sensory locations, we account for the regret associated with the path from one target location to the next.  In many energy-constrained exploration tasks, the regret during the transient phase is important and should not be neglected. For instance, Mars exploration rovers aim to maximize coverage within a limited life-cycle traveling distance; agricultural spray drones must optimize coverage in unsprayed areas within constrained flight time. In both cases, regret during the transient phase demands careful consideration in the algorithmic design.


\vspace{-5pt}
\subsection{Gaussian Processes}
\vspace{-5pt}

Some assumptions on the reward map $w$ are 
required to guarantee no-regret. Here we assume that the reward distribution $w$ is sampled from a Gaussian Process (GP) $\GP(\mu, \kappa)$,\footnote{Our assumption $w \ge 0$ does not conflict with the Gaussian Process (GP) assumption. This is because we have the flexibility to shift the function sampled from the GP, ensuring non-negativity while preserving the same regret.} which is specified by a mean function $\mu: V\to \mathbb{R}$ and a covariance function $\kappa:V\times V \to \mathbb{R}$. It represents a collection of dependent random variables, one for each $v\in V$, every finite subset $s\subseteq V$ of which is multivariate
Gaussian distribution with mean $\mu(s)$ and variance $\kappa(s,s)$. Here for subsets $s,s'\subseteq V$, the term $\kappa(s,s')$ represents a matrix of size $|s|\times|s'|$, where the $ij$-th entry $[\kappa(s,s')]_{i,j} = \kappa(s_i,s'_j)$ ($s_i$ is the $i$-th entry of $s$). For noisy samples $Y(s)= w(s) + \epsilon$ at vertices $s = \{v_1, v_2, \dots,v_k\}$ with i.i.d. Gaussian noice $\epsilon\sim \mathcal{N}(0,\sigma^2I)$ (without causing notational confusion we also represent $Y(s)$ as a $|s|$ dimensional vector where $[Y(s)]_i = Y(s_i)$), the posterior distribution over the reward map $w$ is a GP again, i.e. $w | \left\{s,Y(s)\right\}\sim \GP(\mu',\kappa')$, with mean $\mu'$ and covariance $\kappa'$ given by:
\begin{equation}\label{eq:GP-update}
\begin{split}
\vspace{-10pt}
    \mu'(v) = \mu(v)+\kappa(v,s) (\kappa(s,s) + \sigma^2 I)^{-1}Y(s),\quad 
    \kappa'(u,v) =\kappa(u,v) - \kappa(u, s) (\kappa(s,s) + \sigma^2 I)^{-1} \kappa(s,v)
\vspace{-10pt}
\end{split}
\end{equation}
A more comprehensive discussion of topics related to GP can be found in \citep{williams2006gaussian}.

\vspace{-10pt}
\section{Algorithm Design}
We present a detailed description of the multi-agent coverage control with the doubling trick (MAC-DT, Algorithm \ref{alg:main}), which can be decomposed into the following steps:
\paragraph{Upper Confidence Bound Construction.} Similar to the standard UCB algorithm for Bayesian optimization with GP (c.f. \citep{srinivas2009gaussian}), we keep track of a posterior estimate of the GP and use it to construct a UCB for the reward map $w$ at each episode. At the start of each episode $e$, the mean function $\mu^{(e-1)}$ and the covariance $\kappa^{(e-1)}$ is updated based on the posterior estimation, and the UCB for the reward map is calculated by $w_{\UCB}^{(e)} := \mu^{({e-1})} + \beta^{(e)} \sigma^{({e-1})}$, where $\sigma^{(e-1)}$ is the standard deviation function, i.e. $\sigma^{(e-1)}(v) := \sqrt{\kappa^{(e-1)}(v,v)}$ and $\beta^{(e)}$ is some pre-specified constant.

\paragraph{Destination Selection with $\Oracle$} In each episode $e$, given the UCB reward map $w_{\UCB}^{(e)}$, the algorithm computes the destination of each agent using an oracle algorithm $\Oracle$ which takes a given reward map as its input and outputs the destination of each agent $\{v_{1,dest}, v_{2,dest}, \dots,v_{N,dest}\} = \Oracle(w_{\UCB}^{(e)})$. The most ideal $\Oracle$ would solve the optimal max-coverage problem of the given reward map $w$, i.e. $\Oracle(w) =\argmax_{s = \cup_{i=1}^N s_{i}(v_i), v_i\in V} \|w(s)\|_1$. However, the maximum coverage problem can be NP-hard \citep{feige1998threshold}. Consequently, practical solutions frequently resort to approximation algorithms, such as greedy algorithms, to provide an approximate solution. For this consideration, we introduce the $\alpha$-approximated oracle:
\begin{defi}[$\alpha$-approximated oracle]
    An oracle algorithm $\Oracle$ is called an \textit{$\alpha$-approximated oracle} if for every reward map $w$, the output of the oracle $\{v_{1,dest}, v_{2,dest}, \dots,v_{N,dest}\} = \Oracle(w)$ satisfies:
    \begin{equation*}
       \textstyle  \|w(s_{dest})\|_1\ge \alpha \|w(s^\star)\|_1, \textup{ where   } s_{dest} = \cup_{i=1}^N s_i(v_{i,dest}),~~ s^\star =  \argmax_{s = \cup_{i=1}^N s_{i}(v_i), v_i\in V}\|w(s)\|_1 
    \end{equation*}
\end{defi}
\begin{example}[The greedy oracle] It can be shown that the greedy algorithm oracle is a $(1\!-\!\frac{1}{e})$-approximated oracle \citep{nemhauser1978analysis,krause2014submodular,prajapat2022near}. The $\Oracle$ is defined as
\vspace{-5pt}
\begin{talign*}
        v_{1,dest} = \argmax_{v} w(s_1(v)),\quad  v_{i,dest} = \argmax_{v} w(s_i(v)\backslash \cup_{j=1}^{i-1} s_j(v_{j,dest})), i\ge 2,
    \end{talign*}
    where $A \backslash B$ denotes the relative complement of $B$ in $A$.
\end{example}

\paragraph{Path Planning and Sample Selection} Given the destination output by the $\Oracle$, each agent plans its path on the graph to reach its destination. Along the path, at each time step $t$ agents selects its evaluation point $v_{i,t}^\eval$ such that it is the most uncertain point within its covering region $s_{i,t}$, i.e. $v_{i,t}^\eval = \argmax_{v\in s_{i,t}} \sigma^{(e-1)}(v)$.

\paragraph{The Doubling Trick}
As stated in Line 6 of the algorithm, the termination condition is given by the `doubling trick', i.e.,  at least one of the vertex's samples doubles. If an agent reaches its destination before the episode terminates, it stays at the destination and keeps collecting samples until the end of the episode. If the termination condition is triggered before the agent reaches its target, it promptly terminates its current path and transitions to the next episode.
It is worth noting that the doubling trick serves as a major difference between our algorithm and MACOPT in \citep{prajapat2022near}, which plays an important role in proving sublinear regret considering transient behavior. Similar techniques have also been used for regret analysis in different settings \citep{zhang2023multi,auer2008near}.  Numerical simulations also suggest that the doubling trick can improve and stabilize the coverage performance, especially in the case where the observation is noisy and high reward locations are relatively scattered.
\begin{rmk}
It's essential to note that MAC-DT maintains a partial decentralization approach. Specifically, it requires a centralized coordinator responsible for storing and broadcasting critical global information, such as the UCB of the reward map $w^{(e)}_\UCB$ and $n_v(t)$'s, to each individual agent.
\end{rmk}



\vspace{-10pt}
\begin{algorithm}
\caption{Multi-agent Coverage Control with the Doubling Trick (MAC-DT)}\label{alg:main}
    \begin{algorithmic}[1]
    \REQUIRE An Oracle algorithm $\Oracle$ that calculates agents' destination given a reward map. We also denote $\mu^{(0)} = \mu, \kappa^{(0)} = \kappa, \sigma^{(0)}(v) = \sqrt{\kappa(v,v)}$. We use $t_e$ to represent the total timestep when episode $e$ starts. 
    \FOR{episode $e = 1,2,...$}
        \STATE Calculate the upper confidence bound $w_{\UCB}^{(e)} = \mu^{({e-1})} + \beta^{(e)} \sigma^{({e-1})}$
        \STATE Set the destination using the oracle $\{v_{1,dest}, v_{2,dest}, \dots,v_{N,dest}\} = \Oracle(w_{\UCB}^{(e)})$. \label{algline:oracle}
        \STATE Path Planning: Compute the shortest path on the graph $G$ from current location $v_{i,t_{e}}$ to $v_{i,dest}$. \label{algline: pathplanning}
        \STATE Collect samples: Agents follow their planned paths and at time step $t$, each agent $i$ sets the evaluation point as $v_{i,t}^\eval = \argmax_{v\in s_{i,t}} \sigma^{(e-1)}(v)$, and collect sample $Y(v_{i,t}^\eval)$. When an agent reaches its destination, it stays at the destination and keeps collecting samples until the episode terminates.
        \STATE Episode termination criteria \!-\! the doubling trick: episode terminates at a minimum $t$ such that there exists $v\in V$, such that $n_v(t) \ge \max\{2n_v(t_e-1),1\}$, i.e. when at least one of the vertex's samples doubles. \label{algline:doubling_trick}
        \STATE Update the mean value $\mu^{(e)}$ and $\kappa^{(e)}$ covariance of GP according to \eqref{eq:GP-update} such that $ w\big|D \sim \GP(\mu^{(e)}, \kappa^{(e)}),$ where $D = \{s_{t_e}^\eval,Y(s_{t_e}^\eval),s_{t_e - 1
        }^\eval, Y(s_{t_e - 1
        }^\eval),\dots, s_1^\eval,Y(s_1^\eval)\}$. Set $\sigma^{(e)}(v) = \sqrt{\kappa^{(e)}(v,v)}$
    \ENDFOR
    \end{algorithmic}
\end{algorithm}
\vspace{-10pt}

\section{Main Result}
We first define the following variable which takes the same definition as in \citep{srinivas2009gaussian}:
\vspace{-5pt}
\begin{equation}\label{eq:gamma-def}
   \textstyle \gamma_N:= \max_{s\subseteq V, |s| = N} I(Y(s);w),
    \vspace{-5pt}
\end{equation}
where $I(X;Y)$ denotes the mutual information (c.f. \citep{mackay2003information}) between random variables $X$ and $Y$. This quantity is a frequently employed term in the context of proving regret bounds for Bayesian type algorithms based on GPs (e.g. \citep{srinivas2009gaussian,prajapat2022near}).
Now we state the main theorem:
\begin{theorem}\label{theorem:regret}
   For Algorithm \ref{alg:main} with $\Oracle$ as an $\alpha$-approximated oracle, by setting $\beta^{(e)} \!=\!\! \sqrt{2\log\!\left(|V|\pi^2 e^2/6\delta\right)}$, with probability at least $1-\delta$, the $\alpha$-approximated regret of Algorithm \ref{alg:main} can be bounded by
   \begin{talign*}
       R^\alpha(\texttt{Alg}, T, w)&\le \underbrace{8\sigma \sqrt{\gamma_{2NT}K|V| T\log\left(2|V|\pi^2 T^2/3\delta\right)}}_{\textup{Part I}} \\& \hspace{-30pt}\underbrace{+ nD\K|V|\max_v(\mu(v) +\beta^{(1)} \sqrt{\kappa(v,v)})(\log T + 2)+ 2|V|\sqrt{2NK\log\left(2|V|\pi^2 T^2/3\delta\right)\max_v\kappa(v,v)}}_{\textup{Part II}} 
   \end{talign*}
   where the term $\gamma_{2NT}$ is defined as in \eqref{eq:gamma-def}.
\end{theorem}
\begin{rmk}
The regret consists of a term (Part II) that scales with $O(\log(T))$ and a term that scales with $\widetilde O(\gamma_{2NT}K|V|T)$ (Part I). We mainly focus our discussion on Part I. Note that the term $\gamma_{2NT}$ captures the largest possible mutual information between the samples with size $2NT$ and the true reward map. This quantity generally grows sublinearly with $T$ for commonly used kernels \citep{srinivas2009gaussian}, e.g. for the squared-exponential kernel in the 2-dimension case, $\gamma_{2NT}\sim O((\log(2NT))^3)$, which leads to a final regret of order $\widetilde O(\sqrt{T})$. We would also like to compare our result with the regret bound in \citep{prajapat2022near}, where they bound the regret without accounting for the transition behavior as $\widetilde O(\gamma_{NT}NK|V|T)$. Note that our Part I managed to remove the dependency on $\sqrt{N}$. This major difference arises from a more careful analysis on the bound of the maximum eigenvalue of the covariance matrix (see Remark \ref{rmk:bound-Ke} in the Appendix). It is worth noting that the application of our derived bound to the analysis in \citep{prajapat2022near} has the potential to enhance their regret bound by eliminating the dependency on $\sqrt{N}$ as well.     
\end{rmk}
\subsection{Proof Sketches}
This section provides a brief proof sketch for Theorem \ref{theorem:regret}, which can be decomposed into three major steps. The first step, regret decomposition, breaks down the regret into two terms, namely the `destination switch' and the `price of optimism'. Then the second and third steps bound these two terms respectively.
\paragraph{Regret Decomposition}\label{sec:regret-decomposition} We define the clean event to be $\mu^{({e-1})} - \beta^{(e)} \sigma^{({e-1})}\le w\le = \mu^{({e-1})} + \beta^{(e)} \sigma^{({e-1})} = w_{\UCB}^{(e)}$ (for all $v\in V$), i.e., the true reward map lies within the confidence bound created by $\mu^{({e-1})} \pm \beta^{(e)}\sigma^{({e-1})}$. By carefully selecting the parameters $\beta^{(e)}$, it can be demonstrated that the algorithm will consistently fall within the clean event with a high probability. Thus, for the proof sketch, we will focus solely on the clean event. For simplicity, we also first consider $T$ where $T$ is the last time step of episode $E$, i.e., $T = t_{E+1} -1$.
\begin{talign}
   &\quad \textstyle R^\alpha(\texttt{Alg}, T, w)  =\alpha T\|w(s^\star)\|_1 - \bE \sum_{t=1}^T \|w(s_t))\|_1= \bE \sum_{e=1}^E\sum_{t=t_e}^{t=t_{e+1}-1}[\alpha\|w(s^\star)\|_1- \|w(s_t))\|_1] \notag \\
   &= \textstyle\bE \sum_{e=1}^E\sum_{t=t_e}^{t=t_{e+1}-1}[\alpha\|w(s^\star)\|_1- \|w_{\UCB}^{(e)}(s_t)\|_1] + \bE \sum_{e=1}^E\sum_{t=t_e}^{t=t_{e+1}-1}[\|w_{\UCB}^{(e)}(s_t)\|_1 - \|w(s_t)\|_1]\notag \\
   &\le\textstyle \!\bE \!\sum_{e\!=\!1}^E\sum_{t=t_e}^{t\!=\!t_{e\!+\!1}\!-\!1}\![\alpha\|w(s^\star)\|_1\!\!-\! \|w_{\UCB}^{(e)}(s_t)\|_1] \!+\! \bE\! \sum_{e=1}^E\!\sum_{t=t_e}^{t=t_{e\!+\!1}\!-\!1}[  \|w_{\UCB}^{(e)}(s_t)\|_1\! \!-\!\|[\mu^{({e\!-\!1})} \!-\! \beta^{(e)} \sigma^{({e\!-\!1})}](s_t)\|_1]\notag \\
   &= \underbrace{\textstyle\bE \sum_{e=1}^E\sum_{t=t_e}^{t=t_{e+1}-1}[\alpha\|w(s^\star)\|_1- \|w_{\UCB}^{(e)}(s_t)\|_1]}_{\textup{Destination Switch}} + \underbrace{\textstyle\bE \sum_{e=1}^E2\beta^{(e)}\sum_{t=t_e}^{t=t_{e+1}-1}\|\sigma^{({e-1})}(s_t)\|_1 }_{\textup{Price of Optimism}}\label{eq:regret-decomposition}.
\end{talign}
Here the term `destination switch' measures the regret from visiting sub-optimal nodes during transit to the destination in each episode. The term  `price of optimism' captures the regret from using the UCB as a surrogate for the actual rewards. We now bound each term separately.

\paragraph{Bound the destination switch}\label{sec:destination-switch}
The bound for the destination switch is relatively easy and straight forward:

\begin{lemma}\label{lemma:destination-switch}
   $\!\!\!
        \underbrace{ \textstyle \bE\! \sum_{e\!=\!1}^E\!\!\sum_{t=t_e}^{t=t_{e\!+\!1}\!-\!1}\![\alpha\|w(s^\star)\|_1\!\!-\! \|w_{\UCB}^{(e)}(s_t)\|_1]}_{\textup{Destination Switch}}\!\le\! \alpha nD\K|V|\!\max_v(\mu(v) \!\!+\!\!\beta^{(1)} \!\!\sqrt{\kappa(v,v)})(\log(t_{E\!+\!1}\!-\!1)\!+\!1)\!.$
\begin{proof}
When agents haven't reached the destinations, $\alpha\|w(s^\star)\|_1- \|w_{\UCB}^{(e)}(s_t)\|_1 \le \alpha\|w(s^\star)\|_1\le N\K \max_v w(v)$. Since $w(v)\!\le\! \mu^{(0)}(v) \!+\! \beta^{(1)}\sigma^{(0)}(v)\!\le\! \max_v(\mu(v) \!+\!\beta^{(1)} \sqrt{\kappa(v,v)}) \Rightarrow \alpha\|w(s^\star)\|_1\!- \|w_{\UCB}^{(e)}(s_t)\|_1 \!\le\! NK\max_v(\mu(v) +\beta^{(1)} \sqrt{\kappa(v,v)}) $. When the agents have reached the destination, we have that $s_t = s_{dest} = \Oracle(w_{\UCB}^{(e)})$, and thus $\|w_{\UCB}^{(e)}(s_t)\|_1 \ge \alpha \max_{s}\|w_{\UCB}^{(e)}(s)\|_1 \ge \alpha \|w(s^\star)\|_1 $. Since the diameter of graph $G$ is $D$, it takes the agents at least $D$ steps to reach the destination, thus for each episode, we have that $\sum_{t=t_e}^{t=t_{e+1}-1}[\alpha\|w(s^\star)\|_1- \|w_{\UCB}^{(e)}(s_t)\|_1] \le nD\K\max_v(\mu(v) +\beta^{(1)} \sqrt{\kappa(v,v)})$. Then from Lemma \ref{lemma:bound-episode-number} in Appendix \ref{apdx:auxiliaries} in \citep{supp}, the number of episode $E$ can be bounded by  $E \le |V|\log(t_{E+1}-1) + 1$, the destination switch can then be bounded by ~$\bE \alpha nD\K E \le\alpha nD\K|V|\max_v(\mu(v) +\beta^{(1)} \sqrt{\kappa(v,v)})(\log(t_{E+1}-1)+1)$.
\end{proof}
\end{lemma}
\vspace{-20pt}
\paragraph{Bound the price of optimism}\label{sec:price-of-optimism}
The bound for price of optimism is technically more involved, thus we defer the full proof to Appendix \ref{apdx:bound-price-of-optimism}. The key step is to bound $\|\sigma^{(e-1)}(s_t^\eval)\|_2^2$ using the mutual information of the function evaluations and the true reward map, i.e., $I(Y(s^\eval_{1:t_{E+1}-1});w) $ as stated in the following lemma.
\begin{lemma}(Informal, formal statement see Lemma \ref{lemma:information-gain-to-regret-formal})\label{lemma:information-gain-to-regret}  ~$\sum_{e\!=\!1}^E \!\sum_{t=t_e}^{t_{e\!+\!1}\!-\!1}\! \|\sigma^{({e\!-\!1})}(s^\eval_t)\|_2^2 \!\le\! \frac{2\sigma^{2}}{\log2} I(Y(s^\eval_{1:t_{E\!+\!1}\!-\!1});w)$.
\end{lemma}
Then using Cauchy schwartz inequality we can get the bound:
\begin{lemma}\label{lemma:bound-price-of-optimism}
$\!\!\underbrace{\textstyle 
 \bE\sum_{e\!=\!1}^E\!2\beta^{(e)}\!\sum_{t=t_e}^{t=t_{e\!+\!1}\!-\!1}\!\|\sigma^{\!({e\!-\!1})}\!(s_t)\|_1}_{\textup{Price of Optimism}} \!\!\le\!  4\sigma \beta^{(\!E)}\!\!\sqrt{\!K|V|(t_{E\!+\!1}\!-\!1)\!\gamma_{\!N(t_{E\!+\!1}\!-\!1}) } \!+\! 2\beta^{(E)} |V|\sqrt{NK\max_v\!\kappa(v,\!v)}.$

\end{lemma}

\vspace{-5pt}
\section{Experimental Results}
This section evaluates our algorithm on multiple numerical simulation tasks and a physical multi-robot coverage task. We also present a variant of Algorithm \ref{alg:main} for the multi-agent coverage with safety considerations.
\vspace{-5pt}
\subsection{Numerical Simulations}
 \begin{wrapfigure}{r}{0.3\textwidth}
 \vspace{-30pt}
\centering
\includegraphics[width=.25\textwidth]{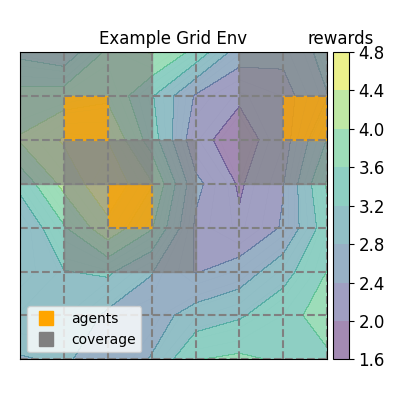}
\vspace{-15pt}
\caption{Grids environment setup.}\label{fig:env}
\vspace{-15pt}
\end{wrapfigure}

\textbf{Environment setup.} We discretize the environments into grids where each agent can cover its 1-hop neighborhood as shown in Figure \ref{fig:env}. We evaluate our algorithm, \texttt{MAC-DT}, with different reward maps, kernel hyperparameters, and observation noises at different scales. Figure \ref{fig:numerical-regret} shows the results with noise variance $0.1$, the first three figures show results with three agents in 8$\times$8 grids with different reward maps, and the last two show 6 and 10 agents in 10$\times$10 grids. The reward maps (listed in titles in Figure \ref{fig:numerical-regret}) are denoted as $w_{\text{Normal}}$, $w_{\text{Uniform}}$, and $w_{\text{Sparse}}$. The first two mean rewards are sampled from Gaussian or uniform distributions, and the last one means only a small number of grids have reward $1$ and others all have reward $0$. We use Gaussian kernels for the GPs. The complete results with 6 and 10 agents, more noise levels, and different kernel hyperparameters can be found in \ref{sec:apdx_exp_texts} in \citet{supp}. 

We compare our algorithm with two baselines. (i) Shortest path planning with MacOpt~\citep{prajapat2022near} named as \texttt{MacOpt-SP}. The only difference between \texttt{MacOpt-SP} and proposed algorithm is the episode termination criteria in line \ref{algline:doubling_trick} in Algorithm \ref{alg:main}. Instead of the doubling trick, \texttt{MacOpt-SP} terminates episodes when all the agents reach their destinations. (ii) A modification of Voronoi partition coverage control from \citet{wei2021multi} named \texttt{Voronoi}. Modifications can be found in Appendix \ref{sec:apdx_exp_texts} in \citep{supp}. Results in Figure \ref{fig:numerical-regret} show that the regret curves of the \texttt{MAC-DT} stay level after a few iterations under all rewards and kernel settings, which means the proposed algorithm quickly finds the no-regret maximal coverage. \texttt{MacOpt-SP} is not efficient since its regret increases faster than MAC-DT, which shows the doubling trick greatly improves the performance. The regret of the Voronoi partition quickly increases in the initial stages since it has an initial sampling stage to reduce the uncertainty globally, which is inefficient when considering the transient behaviors. 

\begin{figure}[ht]
    \centering
    \includegraphics[width=0.19\linewidth]{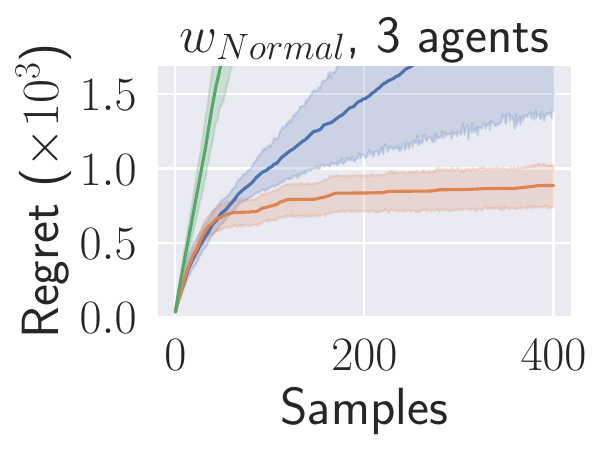}
    \includegraphics[width=0.19\linewidth]{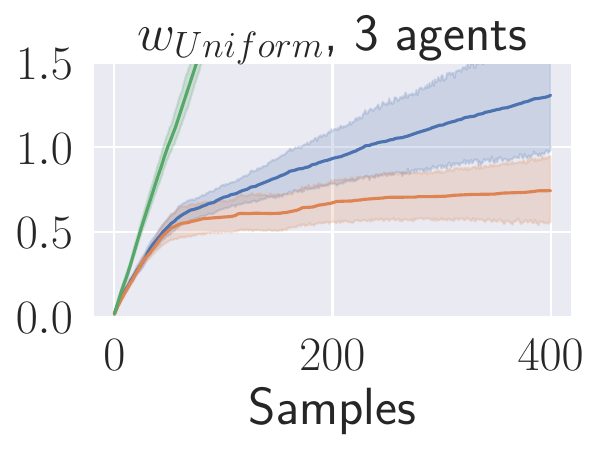}
    \includegraphics[width=0.19\linewidth]{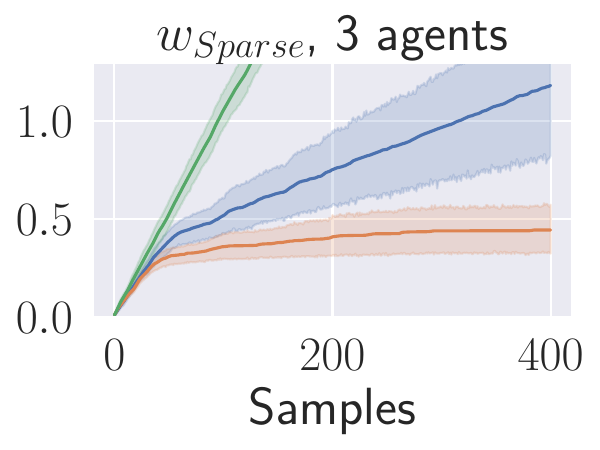}
    \includegraphics[width=0.19\linewidth]{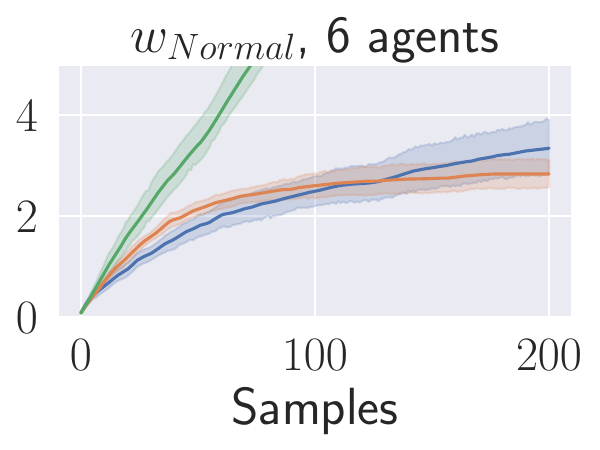}
    \includegraphics[width=0.19\linewidth]{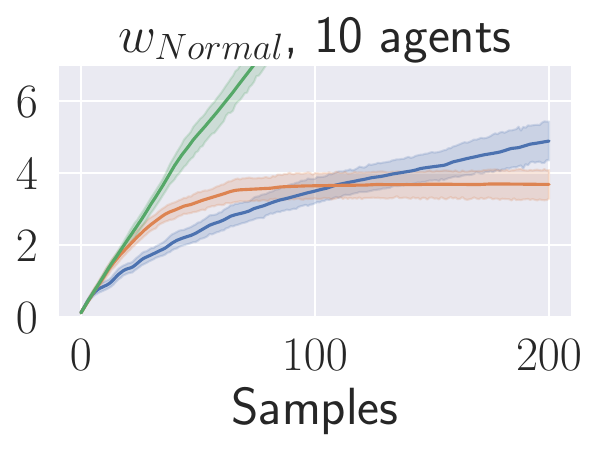}
    \\
    \includegraphics[width=0.6\linewidth]{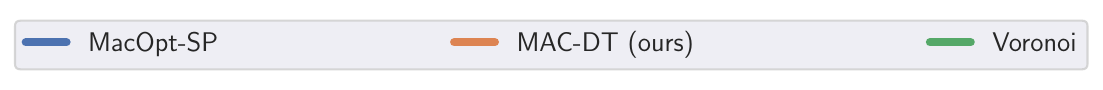}
    \vspace{-10pt}
    \caption{Regret $\textstyle R(\texttt{Alg}, T, w)$ with respect to $T$ with different algorithm $\texttt{Alg}$ rewards $w$. Lines and shaded regions are mean and confidential intervals over 10 randomly generated reward maps.} 
    \label{fig:numerical-regret}
\end{figure}




 \begin{wrapfigure}{r}{0.33\textwidth}
 \vspace{-50pt}
\centering
\includegraphics[width=.25\textwidth]{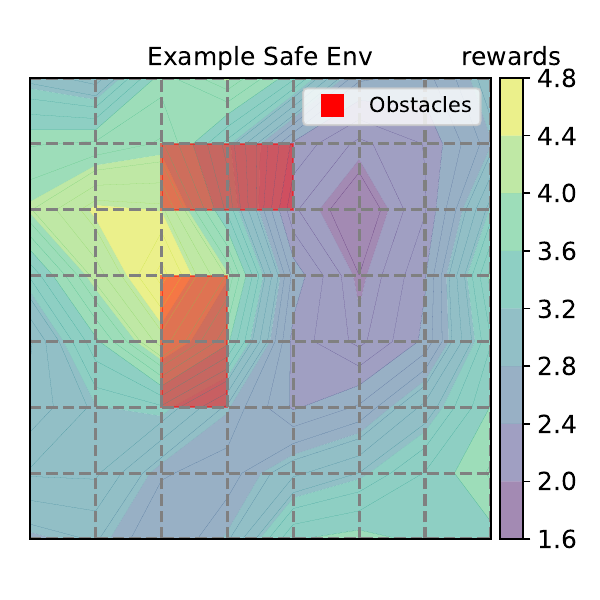}
\vspace{-15pt}
\caption{Example safe exploration environment.}\label{fig:safety}
\end{wrapfigure}
\vspace{-20pt}
\subsection{Safety Considerations}


Safety considerations are common in real-world multi-robot coverage tasks. For example, there are usually obstacles in the environment (like grids painted red in Figure \ref{fig:safety}) that agents should avoid. 
We show that our methods could be combined with safety considerations in this section. We define the set of safe nodes $V_{\text{Safe}}\!=\!\{v\!\in\! V|g(s)\!\geq\! 0\}$ by a safety function $g\!:\!V \!\to\! \mathbb R$. 
The agents should travel within the subgraph containing only nodes in $V_{\text{Safe}}$. We assume there is also uncertainty in the safety function $g$ and agents need to learn  from the noisy samples.  
Similar to the rewards $w(\cdot)$, we maintain a GP for safety function $g$ whose posterior mean and variance at episode $e$ are denoted as $\mu^{(e)}_g(\cdot)$ and $\sigma_g^{(e)}(\cdot)$. 

\begin{wrapfigure}{r}{0.33\textwidth}
    \includegraphics[width=\linewidth]{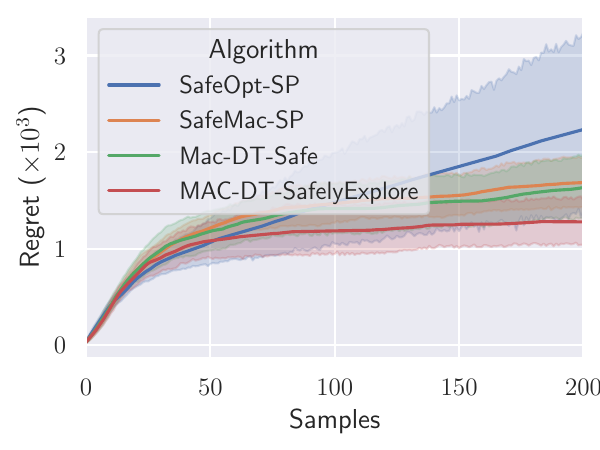}
    \vspace{-15pt}
    \caption{Regret of safe planning algorithm.}
    \label{fig:safe_regret}
    \vspace{-20pt}
\end{wrapfigure}

We made two modifications to Algorithm \ref{alg:main} to consider safety. (i) Changing the oracle and path planning to ensure safety. We use a safe oracle to set the destinations in Line \ref{algline:oracle} following the SafeMac algorithm in \citet{prajapat2022near} that ensures the destinations are safe. The oracle also maintains an estimation of the safe set at episode $e$ denoted by $\hat V^{(e)}_{\text{Safe}}$ and learns to gradually expand it. Then in Line \ref{algline: pathplanning}, the shortest path planning is restricted within the estimated safe set $\hat V^{(e)}_{\text{Safe}}$. (ii) Assigning edge weights during shortest path planning in Line \ref{algline: pathplanning} to encourage safe set expansion. The edges are weighted by the mean value of posterior safety value $\mu^{(e)}_g(\cdot)$. In this way, the agents are prone to travel and collect samples on less safe nodes. These less safe nodes usually lie on the boundaries of the estimated safe set so the agents have more useful samples to expand the safe sets. We name the safety algorithm as \texttt{MAC-DT-SafelyExplore}. 
The algorithm block and details about the safe oracle can be found in Appendix \ref{apdx:safety_algo} in \citet{supp}. 


We compare the \texttt{MAC-DT-SafelyExplore} with three baselines: (a) the ablation study of weighted shortest path planning, which means only adding the safe oracle and restricted path planning to Algorithm \ref{alg:main}. The algorithm is named as $\texttt{Mac-DT-Safe}$; (b) Adding shortest path planning for transient behavior considerations to SafeOpt in \citet{sui2015safe} and SafeMaC in \citet{prajapat2022near}, named \texttt{SafeOpt-SP} and \texttt{SafeMac-SP}. Both two algorithms did not consider transient behaviors originally. 

The regret is shown in Figure \ref{fig:safe_regret}. 
 Results show that \texttt{MAC-DT-SafelyExplore} outperforms all baselines and achieves no-regret coverage quickly after about 50 samples. All the baselines cannot achieve no-regret coverage within 200 samples. The comparison with \texttt{Mac-DT-Safe} shows that the proposed weighted path planning is effective in encouraging safe set expansion. The results also suggest that it is important to consider the transient phase if the problem has safety considerations. 


    

\subsection{Real World Experiments}
\begin{wrapfigure}{r}{0.33\textwidth}   
\includegraphics[width=0.9\linewidth]{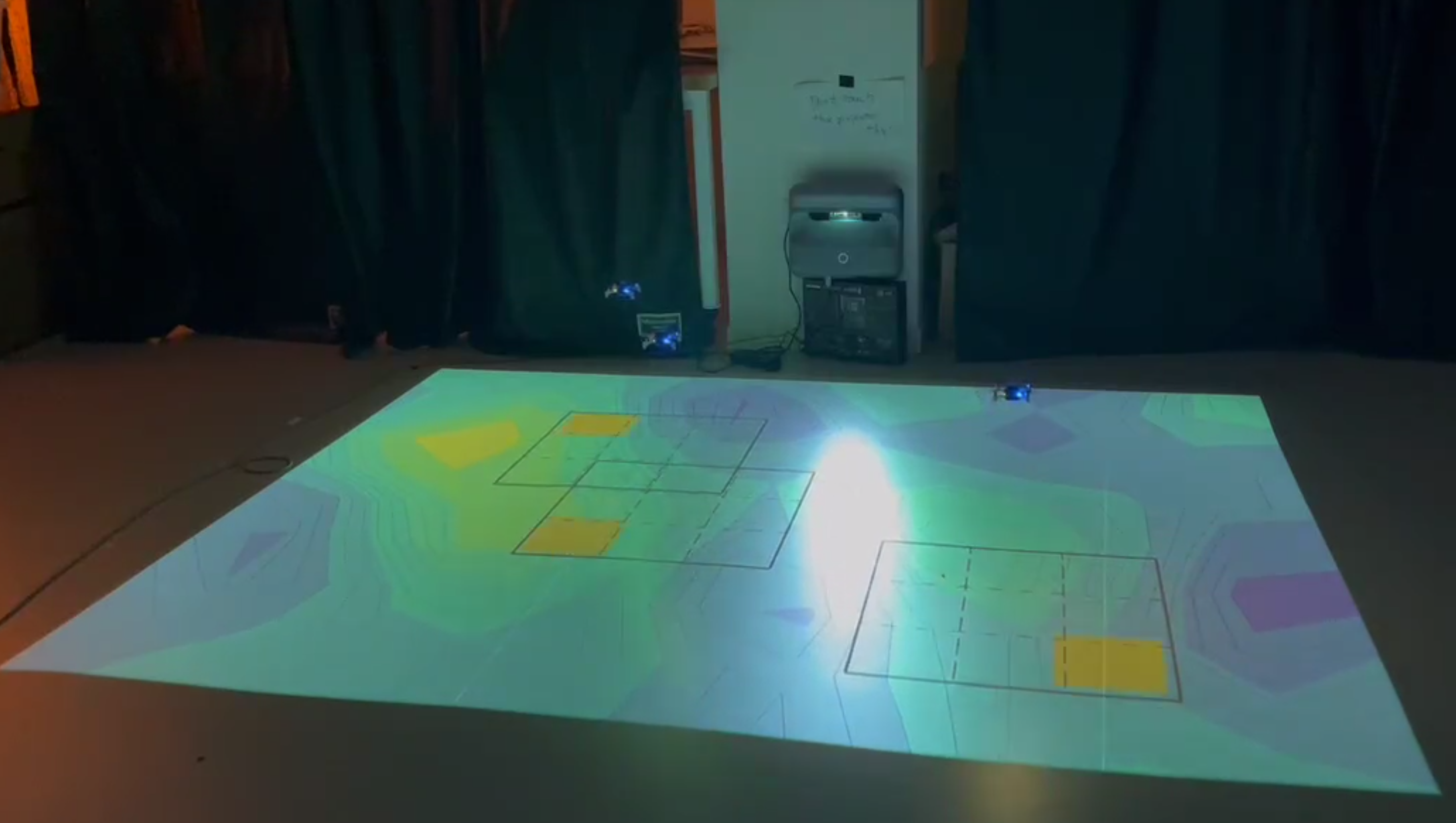}
    \caption{Real-world experiment setup. The projected contour plot indicates the node rewards, and the sensory range and sensory points. }   \label{fig:real_setup}
\end{wrapfigure}
We set up a physical multi-drone coverage testbed as shown in Figure \ref{fig:real_setup}. We fetched the real-time rain data from the OpenWeatherMap and used drones to cover the area with the heaviest rains. We discretize the region into a 13$\times$10 grid. Then we project the real-world weather data on the ground in our lab, where each 17cm square grid represents 1 km$^2$ in the real world.  We use three Crazyflie 2.1 and each can cover the 1-hop neighborhood in the grids. The planning (including collision avoidance) and control are computed onboard the Crazyflie while receiving virtual source signals. The terminal condition is achieving optimal coverage. The experiment video is shown in \url{https://youtu.be/ImgSS5QyhT0}. The three drones find the optimal coverage with only 19 steps and 57 environment samples in total.

\vspace{-5pt}
\section{Conclusion and Future Works}
\vspace{-2pt}
In this paper, we propose the MAC-DT algorithm to solve the MAC problem, which achieves (approximated) regret of $\widetilde O(\sqrt{T})$ even when accounting for the transient behavior. Our results are also supported by numerical studies in multiple settings, showcasing the algorithm's effectiveness and its adeptness in accommodating safety constraints.  Our result admittedly has its limitations. First of all, MAC-DT is not fully decentralized. Secondly, although the simulation results demonstrate that our methods could naturally be combined with safety considerations, our current theoretical analysis doesn't consider safety issues. To address these shortcomings, we envision our future work focusing on the design of decentralized algorithms that not only maintain efficiency but also incorporate robust safety guarantees.



\vspace{-10pt}
\section*{Appendix}
\appendix
\vspace{-5pt}
\textbf{Notations:} We first define some notations that will be useful for the proof. Recall the definition of $n_v(t)$ which denotes the number of times that vertex $v$ has been sampled until time $t$. We define $n_v^{(e)}:= n_v(t_{e+1}-1)$, i.e., $n_v^{(e)}$ is the number of times that grid $u$ is being sampled/covered from $t=1$ to the end of episode $e$. Further, we denote $c_v^{(e)}:= n_v^{(e)} - n_v^{(e-1)}$ as the number of times that grid $v$ is being sampled/covered within episode $e$. 

We define the following matrix $K^{(e)}_{s}:= \kappa^{(e)}(s,s)$, where $s\subseteq V$ is a subset of $V$. Note that for $s$ we allow repetition, i.e. $v\in V$ can appear multiple times in $s$. Additionally, we use $s_{t:\tau} \subseteq V$ to denote the covering profile from time $t$ to $\tau$ (allows repetition).

\section{Proof of Lemma \ref{lemma:information-gain-to-regret} and Lemma \ref{lemma:bound-price-of-optimism}}\label{apdx:bound-price-of-optimism}
We first state some important lemmas that will be used in the proofs.
\begin{lemma}(Quantify the information gain)\label{lemma:quatify-information-gain} ~~$I(Y(s_{1:t_{e+1}-1}^\eval);w) = \sum_{e'=1}^e \frac{1}{2}\log\left(\det\left( I + \sigma^{-2}K^{(e'-1)}_{s^\eval_{t_{e'}:t_{e'+1}-1}}\right)\right)$
\begin{proof} We use $H(X)$ to denote the entropy of a random variable $X$. From the definition of mutual information and properties of entropy function (c.f. \citep{mackay2003information}), we have
    \begin{talign*}
         &\quad I(Y(s^\eval_{1:t_{e+1}-1});w) = H(Y(s^\eval_{1:t_{e+1}-1})) - H(Y(s^\eval_{1:t_{e+1}-1})|w)\\
         &= H(Y(s^\eval_{1:t_{e}-1})) + H(Y(s^\eval_{t_e:t_{e+1}-1})|Y(s^\eval_{1:t_{e}-1})) - \left(H(Y(s^\eval_{1:t_{e}-1})|w)+H(Y(s^\eval_{t_e:t_{e+1}-1})|w) \right)\\
         &= I(Y(s^\eval_{1:t_{e}-1})|w) + H(Y(s^\eval_{t_e:t_{e+1}-1})|Y(s^\eval_{1:t_{e}-1})) - H(Y(s^\eval_{t_e:t_{e+1}-1})|w)
    \end{talign*}
    Further, 
    \vspace{-25pt}
    \begin{talign*}
        &\quad H(Y(s^\eval_{t_e:t_{e+1}-1})|Y(s^\eval_{1:t_{e}-1})) = \frac{1}{2}\log\left(\det\left(2\pi e\left(\sigma^2 I + K^{(e-1)}_{s^\eval_{t_e:t_{e+1}-1}}\right)\right)\right)\\
        &= \frac{1}{2}\log\left(2\pi e(\sigma)^{2\sum_{t=t_e}^{t_{e+1}-1}|s^\eval_t|}\right) + \frac{1}{2}\log\left(\det\left( I + \sigma^{-2}K^{(e-1)}_{s^\eval_{t_e:t_{e+1}-1}}\right)\right)
    \end{talign*}
Additionally, given that $H(Y(s^\eval_{t_e:t_{e+1}-1})|w) = \frac{1}{2}\log\left(2\pi e(\sigma)^{2\sum_{t=t_e}^{t_{e+1}-1}|s^\eval_t|}\right)$,
    we have
    \vspace{-10pt}
    \begin{talign*}
      I(Y(s^\eval_{1:t_{e+1}-1});w) = I(Y(s^\eval_{1:t_{e}-1})|w) + \frac{1}{2}\log\left(\det\left( I + \sigma^{-2}K^{(e-1)}_{s^\eval_{t_e:t_{e+1}-1}}\right)\right).
    \end{talign*}
    \vspace{-10pt}
    Applying this equality iteratively completes the proof.
\end{proof}
\end{lemma}

\begin{lemma}\label{lemma:bound-Ke}
If for all $v\in s_{t_{e+1}-1}$, $n_v^{(e-1)} \ge 1$, then $ \lmax\left(K^{(e-1)}_{s_{t_{e}:t_{e+1}-1}}\right)\le \sigma^2$.

otherwise, there exists at least one $v\in s_{t_{e+1}-1}$, such that $n_v^{(e-1)} = 0$, then $\lmax\left(K^{(e-1)}_{s_{t_{e}:t_{e+1}-2}}\right)\le \sigma^2$
\end{lemma}
\begin{rmk}\label{rmk:bound-Ke}
Lemma \ref{lemma:bound-Ke} serves as one of the technical novelty of our work. Notably, compared with \citep{prajapat2022near} where they directly bound the maximum value of the covariance matrix by the trace of the matrix, i.e., $\lmax\left(K^{(e-1)}_{s_{t_{e}:t_{e+1}-1}}\right) \le \textup{trace}\left(K^{(e-1)}_{s_{t_{e}:t_{e+1}-1}}\right) \sim O(N)$, we bound the maximum eigenvalue in a more careful manner and thus eliminates the dependency on the number of agent $N$, thereby removing the $\sqrt{N}$ dependency in our regret bound. The proof is technically involved and is  deferred to Appendix \ref{apdx:bound-Ke} in \citep{supp}.   
\end{rmk}

\begin{lemma}\label{lemma:log-det-to-trace}
   \! If $~\forall v\!\in\! s_{t_{e\!+\!1}-1}$, $n_v^{(e\!-\!1)} \!\ge\! 1$, then $ \frac{\sigma^2}{\log 2}  \log\left(\det\left( I \!+\! \sigma^{-2}K^{(e-1)}_{s^\eval_{t_{e}:t_{e+1}-1}}\right)\right)\ge \sum_{t=t_e}^{t_{e+1}-1}\|\sigma^{({e-1})}(s^\eval_t)\|_2^2,$
    otherwise, $\frac{\sigma^2}{\log 2}\log\left(\det\left( I + \sigma^{-2}K^{(e-1)}_{s^\eval_{t_{e}:t_{e+1}-1}}\right)\right)\ge \sum_{t=t_e}^{t_{e+1}-2}\|\sigma^{({e-1})}(s^\eval_t)\|_2^2 $.

    \begin{proof}
    For the first case, from Lemma \ref{lemma:bound-Ke} we have that $\lambda\le \sigma^2$ for all $\lambda \in  \lambda\left(K^{(e-1)}_{s^\eval_{t_{e}:t_{e+1}-1}}\right)$, then
      \begin{talign*}
         &\quad \log\left(\det\left( I + \sigma^{-2}K^{(e-1)}_{s^\eval_{t_{e}:t_{e+1}-1}}\right)\right)  =  \sum_{\lambda\in \lambda\left(K^{(e-1)}_{s^\eval_{t_{e}:t_{e+1}-1}}\right)}\log\left(\det\left( I + \sigma^{-2}\lambda\right)\right)\\
    &\ge \sigma^{-2}\log2\sum_{e=1}^E  \sum_{\lambda\in \lambda\left(K^{(e-1)}_{s^\eval_{t_{e}:t_{e+1}-1}}\right)}\lambda ~~\textup{(} \lambda \le \sigma^2\textup{ by Lemma \ref{lemma:bound-Ke} and that } \log(1+x) \ge (\log 2)x, 0\le x\le 1 \textup{)} \\
    & = \sigma^{-2}\log2\sum_{e=1}^E\tr\left(K^{(e-1)}_{s^\eval_{t_{e}:t_{e+1}-1}}\right) =\sigma^{-2}\log2\sum_{e=1}^E\sum_{t=t_e}^{t_{e+1}-1}\|\sigma^{({e-1})}(s^\eval_t)\|_2^2.
      \end{talign*}  
      For the second case, i.e.,  there exists at least one $v\in s_{t_{e+1}-1}$, such that $n_v^{(e-1)} = 0$, we first argue that for all $v\in s_{t_e:t_{e+1}-2}$, $n_v^{(e-1)} \ge 1$. This can be easily verified by contradiction. If there exists at least one $v\in s_{t_e:t_{e+1}-2}$ such that $n_v = 0$, then given the doubling trick in the algorithm, the episode should terminate at the time step when $v$ is sampled, which contradicts the fact that it terminates at $t_{e+1}-1$. Then from Lemma 
      \ref{lemma:bound-Ke} we have that $\lambda\le \sigma^2$ for all $\lambda \in  \lambda\left(K^{(e-1)}_{s^\eval_{t_{e}:t_{e+1}-2}}\right)$, and similarly, $\log\left(\det\left( I + \sigma^{-2}K^{(e-1)}_{s^\eval_{t_{e}:t_{e+1}-2}}\right)\right) \ge\sigma^{-2}\log2\sum_{e=1}^E\sum_{t=t_e}^{t_{e+1}-2}\|\sigma^{({e-1})}(s^\eval_t)\|_2^2.$
      
   \noindent Since $K^{(e-1)}_{s^\eval_{t_{e}:t_{e\!+\!1}\!-\!2}}\!\!\!\!\!$ is a principal submatrix of $K^{(e-1)}_{s^\eval_{t_{e}:t_{e\!+\!1}\!-\!1}}\!\!\!,$ from Lemma \ref{lemma:log-det-principal-submatrix} in \citep{supp}, we have that\\
      $ \log\!\left(\!\det\left(\! I\! +\! \sigma^{-2}K^{(e-1)}_{s^\eval_{t_{e}:t_{e+1}-1}}\!\right)\!\right)   \!\ge \log\!\left(\!\det\!\left(\! I \!+\! \sigma^{-2}K^{(e-1)}_{s^\eval_{t_{e}:t_{e+1}-2}}\!\right)\!\right)\! \ge\!\sigma^{\!-2}\!\log2\sum_{t=t_e}^{t_{e\!+\!1}-2}\|\sigma^{({e\!-\!1})}(s^\eval_t)\|_2^2.$
    \end{proof}
\end{lemma}
\vspace{-10pt}
We are now ready to state and prove the formal version of Lemma \ref{lemma:information-gain-to-regret}.
\begin{lemma}[Formal version of Lemma \ref{lemma:information-gain-to-regret}]\label{lemma:information-gain-to-regret-formal}
    \begin{talign*}
        \sum_{e=1}^E\!\left(\sum_{t=t_e}^{t=t_{e+1}-2}\|\sigma^{({e-1})}(s_t^\eval)\|_2^2 + \mathbf{1}\{\forall v\!\in \!s_{t_{e\!+\!1}\!-\!1}, n_v^{(e-1)} \!\ge\! 1\}\|\sigma^{({e-1})}(s_t^\eval)\|_2^2\right)\le \frac{2\sigma^{2}}{\log2} I(Y(s^\eval_{1:t_{E+1}-1});w) 
    \end{talign*}
\end{lemma}

\begin{proof}[Proof of Lemma \ref{lemma:information-gain-to-regret-formal}]
From Lemma \ref{lemma:quatify-information-gain} and \ref{lemma:log-det-to-trace} we have
\begin{talign*}
    &\quad I(Y(s^\eval_{1:t_{E+1}-1});w) = \sum_{e=1}^E \frac{1}{2}\log\left(\det\left( I + \sigma^{-2}K^{(e-1)}_{s^\eval_{t_{e}:t_{e+1}-1}}\right)\right)\\
    &\ge\frac{\sigma^{-2}\log2}{2}\sum_{e=1}^E\left(\sum_{t=t_e}^{t=t_{e+1}-2}\|\sigma^{({e-1})}(s_t^\eval)\|_2^2 + \mathbf{1}\{\forall v\in s_{t_{e+1}-1}, n_v^{(e-1)} \ge 1\}\|\sigma^{({e-1})}(s_t^\eval)\|_2^2\right). \vspace{-50pt}
\end{talign*} 
\end{proof}\vspace{-20pt}
\begin{proof}[Proof of Lemma \ref{lemma:bound-price-of-optimism}]
    From Cauchy Schwarz inequality and that $\beta^{(e+1)}\ge \beta^{(e)}$,
    \begin{talign*}
    & \!\!\left(\bE\sum_{e=1}^E2\beta^{(e)}\sum_{t=t_e}^{t=t_{e\!+\!1}\!-\!1}\|\sigma^{({e\!-\!1})}\!(s_t)\|_1\right)^{2}\!\!\!\le\! \bE 4(\beta^{(E)})^2 \!\left(\sum_{e\!=\!1}^E\sum_{t=t_e}^{t=t_{e\!+\!1}\!-\!1}\!|s_t|\right)\left(\sum_{e=1}^E\sum_{t=t_e}^{t=t_{e\!+\!1}\!-\!1}\|\sigma^{({e\!-\!1})}(s_t)\|_2^2\right) \\
    &\le \bE 4(\beta^{(E)})^2 |V|(t_{E+1}-1)\left(K\underbrace{\sum_{e=1}^E\sum_{t=t_e}^{t=t_{e+1}-1}\|\sigma^{({e-1})}(s_t^\eval)\|_2^2}_{\textup{Part A}}\right)
\end{talign*}
Since
\vspace{-15pt}
\begin{talign*}
    & \qquad\qquad\qquad \textup{Part A} \!=\! \sum_{e\!=\!1}^E\!\!\left(\sum_{t=t_e}^{t\!=\!t_{e\!+\!1}\!-\!2}\|\sigma^{({e\!-\!1})}\!(s_t^\eval)\|_2^2 \!+\! \mathbf{1}\{\forall v\!\in\! s_{t_{e+1}-1}, n_v^{(e\!-\!1)} \!\ge\! 1\}\|\sigma^{({e\!-\!1})}(s_t^\eval)\|_2^2\right) \\
    & + \sum_{e=1}^E\left(\mathbf{1}\{\exists v\in s_{t_{e+1}-1}^\eval, n_v^{(e-1)} =0\}\|\sigma^{({e-1})}(s_t^\eval)\|_2^2\right)
    \le  \frac{2\sigma^{2}}{\log2} I(Y(s^\eval_{1:t_{E+1}-1});w)  + N|V|\max_{v\in V}\kappa(v,v) 
\end{talign*}
where the last inequality comes from Lemma \ref{lemma:information-gain-to-regret-formal} and the fact that $\sigma^{(0)}(v) = \max_{v\in V}\sqrt{\kappa(v,v)}$.
Thus 
\begin{talign*}
    \bE\sum_{e\!=\!1}^E\!2\beta^{(e)}\sum_{t=t_e}^{t\!=\!t_{e\!+\!1}\!-\!1}\!\|\sigma^{({e-1})}(s_t)\|_1 &\!\le\! 2\beta^{(E)}\sqrt{K|V|(t_{E\!+\!1}\!-\!1)}\sqrt{\frac{2\sigma^{2}}{\log2} I(Y(s^\eval_{1:t_{E+1}-1});w)  + N|V|\max_{v}\!\kappa(v,v) }\\
    &\!\le\!4\sigma\beta^{(E)}\sqrt{K|V|(t_{E\!+\!1}\!-\!1)\gamma_{N(t_{E+1}-1)}} + 2\beta^{(E)} |V|\sqrt{nK\max_v\!\kappa(v,v)}
\end{talign*}
\end{proof}

\vspace{-35pt}
\section{Proof of Theorem \ref{theorem:regret}}
Before proving the theorem, we first state the following lemma that suggests that with probability at least $1-\delta$, the algorithm will operate on the `clean events':
\begin{lemma}[Lemma 5.1 in \citep{srinivas2009gaussian}]\label{lemma:clean-event-probability}
    By setting $\beta^{(e)} = \sqrt{2\log\left(|V|\pi^2 e^2/6\delta\right)}$, then \\
    $  \textstyle  \left|w(v) -\mu^{(e-1)}(v)\right|\le \beta^{(e)}\sigma^{(e-1)}(v), ~~\forall v\in V, \forall e = 1,2,\dots,$
    holds with probability at least $1-\delta$.
\end{lemma}

\begin{proof}[Proof of Theorem \ref{theorem:regret}]
We set $E$ to be such that $t_{E} < T \le t_{E+1}-1$.
    From Lemma \ref{lemma:clean-event-probability} and \eqref{eq:regret-decomposition}, we have that by setting $\beta^{(e)} = \sqrt{2\log\left(|V|\pi^2 e^2/6\delta\right)}$, then with probability $1-\delta$
    \begin{talign*}
      &\quad R^\alpha(\texttt{Alg}, T, w) \le \bE_{t_E < T\le t_{E+1}-1}R^\alpha(\texttt{Alg}, t_{E+1}-1, w)
      \\ &\le  \underbrace{\bE \sum_{e=1}^E\sum_{t=t_e}^{t=t_{e+1}-1}[\alpha\|w(s^\star)\|_1- \|w_{\UCB}^{(e)}(s_t)\|_1]}_{\textup{Destination Switch}} + \underbrace{\bE \sum_{e=1}^E\sum_{t=t_e}^{t=t_{e+1}-1}2\beta^{(e)}\|\sigma^{({e-1})}(s_t)\|_1]}_{\textup{Price of Optimism}}\\
       &\stackrel{\textup{(Lemma \ref{lemma:destination-switch} and \ref{lemma:bound-price-of-optimism})}}{\hspace{-15pt}\le} \bE nD\K|V|\max_v(\mu(v) +\beta^{(1)} \sqrt{\kappa(v,v)})(\log(t_{E\!+\!1}\!-\!1)\!+\!1)\\
       &\quad + 4\sigma\beta^{(E)}\sqrt{K|V|(t_{E\!+\!1}\!-\!1)\gamma_{N(t_{E+1}-1)}} + 2\beta^{(E)} |V|\sqrt{nK\max_v\kappa(v,v)}
    \end{talign*}
    From the doubling trick, we know that $t_{E+1}-1\le 2T$. Thus
    \begin{talign*}
       &R^\alpha(\texttt{Alg}, T, w) \le \bE nD\K|V||\max_v(\mu(v) +\beta^{(1)} \sqrt{\kappa(v,v)})(\log(2T)+1)  \\
       &\quad\qquad \qquad\qquad\qquad+ 4\sigma\beta^{(E)}\sqrt{2K|V|T\gamma_{2nT}} + 2\beta^{(E)} |V|\sqrt{nK\max_v\kappa(v,v)}\\
        &\!= \!8\sigma \!\sqrt{\gamma_{2nT}K|V| T\log\left(2|V|\pi^2 T^2/3\delta\right)} \\
        &\quad \! +\! nD\K|V|\max_v(\mu(v) +\beta^{(1)} \sqrt{\kappa(v,v)})(\log T \!\!+\! 2) \!+\! 2|V|\!\sqrt{2nK\log\left(2|V|\pi^2 T^2/3\delta\right)\max_v\kappa(v,\!v)},
    \end{talign*}
    which completes the proof.
\end{proof}

\section{Proof of Lemma \ref{lemma:bound-Ke}}\label{apdx:bound-Ke}


Before proving the lemma, we first define some auxiliary variables. We define a new mean function $\wk: V\times V\to \mathbb{R}$ such that for $u,v \in V$
\begin{align*}
    \wk(u,v) &= 0, ~~\textup{if}~ u\neq v\\
    \wk(u,u) &= \sum_{v\in V}\wk(u,v).
\end{align*}

We also define $\wk^{(e)}$ as such that $\wk^{(e)}$ is the covariance matrix of the GP given by $ w~\big|~Y(s_{t_e}), Y(s_{t_e - 1}),\dots,Y(s_1)$, when assuming that the prior distribution of $w$ is sampled from $w\sim \GP(\mu, \wk)$.

Correspondingly, given a set (allows repetition) $s$ we define $\wK^{(e)}_{s}:= \wk^{(e)}(s,s), \wK_s = \wk(s,s)$. We also denote $\wK^{e}:= \wk^{(e)}(V,V), \wK := \wk(V,V)$

\begin{lemma}\label{lemma:wK-ge-K}
    \begin{equation*}
        K^{(e-1)} \preceq \wK^{(e-1)}
    \end{equation*}
\begin{proof}
    From a definition of $\wk$ we have that $\wK - K$ is a diagonally-dominant matrix, i.e. $\wK \succeq K$. Thus $\wK_{\{V,s_{1:t_e-1}\}}\succeq K_{\{V,s_{1:t_e-1}\}}$.
    From standard GP derivation, we get that $K^{(e-1)}, \wK^{(e-1)}$ are the Schur complements of  $\wK_{\{V,s_{1:t_e-1}\}}, K_{\{V,s_{1:t_e-1}\}}$, i.e.
    \begin{align*}
    \wK^{(e-1)} &= \wK - \wk(V, s_{1:t_e-1}) (\wK_{s_{1:t_e-1}}+\sigma^2 I)^{-1}\wk(s_{1:t_e-1},V)\\
        K^{(e-1)} &= K - \kappa(V, s_{1:t_e-1}) (K_{s_{1:t_e-1}}+\sigma^2 I)^{-1}\kappa(s_{1:t_e-1},V).
    \end{align*}
    Then $K^{(e-1)} \preceq \wK^{(e-1)}$ immediately holds true given Lemma \ref{lemma:schur-complement-ineq}.
\end{proof}
\end{lemma}

\begin{lemma}\label{lemma:bound-wKe}
If for all $v\in s_{t_{e+1}-1}$, $n_v^{(e-1)} \ge 1$, then
    \begin{align*}
        \lmax\left(\wK^{(e-1)}_{s_{t_{e}:t_{e+1}-1}}\right)\le \sigma^2
    \end{align*}
otherwise, there exists at least one $v\in s_{t_{e+1}-1}$, such that $n_v^{(e-1)} = 0$, then
\begin{align*}
   \lmax\left(\wK^{(e-1)}_{s_{t_{e}:t_{e+1}-2}}\right)\le \sigma^2
\end{align*}
\begin{proof}
    Since $\wK$ is a diagonal matrix, $\wK^{(e-1)}$ is also a diagonal matrix. For any element $u\in V$, using standard Bayesian law we get that 
    \begin{equation*}
        \wk^{(e-1)}(u,u) = \frac{1}{\sigma^{-2}n_u^{(e-1)} + \wk(u,u)^{-1}},
    \end{equation*}
    where $n_u^{(e)}, c_u^{(e)}$ are defined in the very beginning of the Appendix. Then by reorganizing the columns and rows of $\wK^{(e-1)}_{s_{t_{e}:t_{e+1}-1}}$ we have that
    \begin{align*}
        \lmax\left(\wK^{(e-1)}_{s_{t_{e}:t_{e+1}-1}}\right) &= \lmax\left(\blkdiag\{\wk^{(e-1)}(u,u) E_{c_u^{(e)}}\}\right) = \max_u \wk^{(e-1)}(u,u) c_u^{(e)}\\
        &=\max_u \frac{c_u^{(e)}}{\sigma^{-2}n_u^{(e-1)} + \wk(u,u)^{-1}},
    \end{align*}
   where the matrix $E_n$ is defined as an $n\times n$ dimensional all-one matrix, i.e. $E_n = \left[\begin{array}{ccc}
       1 & \cdots & 1 \\
       \vdots &  &\vdots\\
       1 &\cdots &1
   \end{array}\right]_{n\times n}$.
   
For the first case, i.e. for all $v\in s_{t_{e+1}-1}$, $n_v^{(e-1)} \ge 1$, from the doubling trick, we get that for any $n_u^{(e-1)}\ge 1$, $c_u^{(e)}\le n_u^{(e-1)}$. Then
    \begin{align*}
        \lmax\left(\wK^{(e-1)}_{s_{t_{e}:t_{e+1}-1}}\right) &= \max_{u\in s_{t_{e}:t_{e+1}-1}}\frac{c_u^{(e)}}{\sigma^{-2}n_u^{(e-1)} + \wk(u,u)^{-1}} \le \max_{u\in s_{t_{e}:t_{e+1}-1}} \frac{n_u^{(e-1)}}{\sigma^{-2}n_u^{(e-1)} + \wk(u,u)^{-1}}\\
        &\le \max_{u\in s_{t_{e}:t_{e+1}-1}} \frac{n_u^{(e-1)}}{\sigma^{-2}n_u^{(e-1)}} \le \sigma^2
    \end{align*}

Similar argument holds for the second case, where there exists at least one $v\in s_{t_{e+1}-1}$, such that $n_v^{(e-1)} = 0$. Using the same argument in the proof of Lemma \ref{lemma:log-det-to-trace} we have that in this case, for all $v\in s_{t_e:t_{e+1}-2}$, $n_v^{(e-1)} \ge 1$, then similarly
\begin{align*}
        \lmax\left(\wK^{(e-1)}_{s_{t_{e}:t_{e+1}-2}}\right) &= \max_{u\in s_{t_{e}:t_{e+1}-2}}\frac{c_u^{(e)}}{\sigma^{-2}n_u^{(e-1)} + \wk(u,u)^{-1}} \le \max_{u\in s_{t_{e}:t_{e+1}-2}} \frac{n_u^{(e-1)}}{\sigma^{-2}n_u^{(e-1)} + \wk(u,u)^{-1}}\\
        &\le \max_{u\in s_{t_{e}:t_{e+1}-2}} \frac{n_u^{(e-1)}}{\sigma^{-2}n_u^{(e-1)}} \le \sigma^2,
    \end{align*}
    which completes the proof.
\end{proof}
\end{lemma}

\begin{proof}[Proof of Lemma \ref{lemma:bound-Ke}]
    Lemma \ref{lemma:bound-Ke} is a direct corollary of Lemma \ref{lemma:wK-ge-K} and Lemma \ref{lemma:bound-wKe}. From \ref{lemma:wK-ge-K} we have that $\wK^{(e-1)}_{s_{t_{e}:t_{e+1}-1}} \succeq K^{(e-1)}_{s_{t_{e}:t_{e+1}-1}}$, $\wK^{(e-1)}_{s_{t_{e}:t_{e+1}-1}} \succeq K^{(e-1)}_{s_{t_{e}:t_{e+1}-1}}$. Thus ,
    \begin{align*}
      \lmax\left(K^{(e-1)}_{s_{t_{e}:t_{e+1}-1}}\right) \le  \lmax\left(\wK^{(e-1)}_{s_{t_{e}:t_{e+1}-1}}\right), \quad \lmax\left(K^{(e-1)}_{s_{t_{e}:t_{e+1}-2}}\right) \le  \lmax\left(\wK^{(e-1)}_{s_{t_{e}:t_{e+1}-2}}\right)
    \end{align*}
    Lemma \ref{lemma:bound-Ke} is the proved directly from Lemma \ref{lemma:bound-wKe}.
\end{proof}
\section{Auxiliaries}\label{apdx:auxiliaries}
\begin{lemma}[Bounding the Episode length]\label{lemma:bound-episode-number} For Algorighm \ref{alg:main}, we have that the episode length $E$ can be bounded by
    \begin{equation*}
        E \le |V|\log(t_{E+1}-1) + 1
    \end{equation*}
\begin{proof}
Recall that in Algorithm \ref{alg:main} (Line 6) the termination condition is given by the doubling trick, i.e., there exists $v\in V$ such that $n_v \ge \max{2n_v(t_e), 1}$. We call the above $v$ which triggers the termination condition in at episode $e$ as the `critical vertex' for episode $e$. Now, we count the number of times each vertex is the `critical vertex' for $e = 1,2,\dots, E$. According to the pigeonhole principle, there exists a $v^\star$ such that it is the critical vertex for more than $\frac{E}{|V|}$ times. 

Additionally, given the doubling trick, if vertex $v$ is the critical vertex for episode $e$, then $n_v^{(e)} = 2n_v^{(e-1)}$ if $n_v^{(e-1)} \ge 1$, $n_v^{(e)} = 1$ if $n_v^{(e-1)} = 0$. Thus for vertex $v^\star$, since it is the critical vertex for at least $\frac{E}{|V|}$ episodes, we have that
\begin{align*}
    n_{v^\star}^{(E)} \ge 2^{\frac{E}{|V|}-1}.
\end{align*}
Additionally, given that $n_{v^\star}^{(E)} \le t_{E+1}-1$, we have that
\begin{align*}
    2^{\frac{E}{|V|}-1} \le t_{E+1}-1 ~~\Longrightarrow~~E \le |V|\log(t_{E+1}-1) + 1
\end{align*}
\end{proof}
\end{lemma}

\begin{lemma}\label{lemma:schur-complement-ineq}
Suppose $X = \left[\begin{array}{cc}
    X_{11} & X_{12} \\
    X_{12}^\top & X_{22}
\end{array}\right]$ and $\wX = \left[\begin{array}{cc}
    \wX_{11} & \wX_{12} \\
    \wX_{12}^\top & \wX_{22}
\end{array}\right]$ are symmetric and positive definite matrices, and that $\wX \succeq X$, then the schur complements satisfy the following inequality
\begin{equation*}
    \wX_{11} - \wX_{12}\wX_{22}^{-1}\wX_{12}^\top\succeq X_{11} - X_{12}X_{22}^{-1}X_{12}^\top
\end{equation*}
\begin{proof}
    Let $\wDelta_{11}:= \wX_{11} - \wX_{12}\wX_{22}^{-1}\wX_{12}^\top, \Delta_{11}:= X_{11} - X_{12}X_{22}^{-1}X_{12}^\top$. Then
    \begin{align*}
       \wX = \left[\begin{array}{cc}
    \wDelta_{11} + \wX_{12}\wX_{22}^{-1}\wX_{12}^\top & \wX_{12} \\
    \wX_{12}^\top & \wX_{22}
\end{array}\right], X = \left[\begin{array}{cc}
    X_{11} + X_{12}X_{22}^{-1}X_{12}^\top & X_{12} \\
    X_{12}^\top & X_{22}
\end{array}\right]. 
    \end{align*}
    Thus
    \begin{align*}
        \wX - X = \left[\begin{array}{cc}
    (\wDelta_{11} - \Delta_{11}) + (\wX_{12}\wX_{22}^{-1}\wX_{12}^\top- X_{12}X_{22}^{-1}X_{12}^\top) & \wX_{12} - X_{12} \\
    (\wX_{12} - X_{12})^\top & \wX_{22} - X_{22}
\end{array}\right].
    \end{align*}

    We first consider the case where $\wX_{22} - X_{22}$ is invertible. Since $\wX - X\succeq 0$, its Schur complement should be positive semidefinite, i.e.
    \begin{align}
        (\wDelta_{11} - \Delta_{11})+ \underbrace{(\wX_{12}\wX_{22}^{-1}\wX_{12}^\top- X_{12}X_{22}^{-1}X_{12}^\top) - (\wX_{12} - X_{12})(\wX_{22} - X_{22})^{-1}(\wX_{12} - X_{12})^\top}_{\textup{Part A}}\succeq 0.\label{eq:schur-complement}
    \end{align}
    Thus to prove the statement it is sufficient to show that $\textup{Part A} \preceq 0$.

    From standard linear algebra result we know that there exists an invertible matrix $G$ that simultaneously diagonalizes $\wX_{22}$ and $X_{22}$, i.e.
    \begin{align*}
        \wX_{22} = G^\top \wLambda G,\quad X_{22} = G^\top \Lambda G,
    \end{align*}
    where $\wLambda, \Lambda$ are diagonal matrices. We define $\wY_{12}:=\wX_{12}G^{-1}, Y_{12}:=X_{12}G^{-1}$. Then
    \begin{align*}
        &\quad \textup{Part A} = \wY_{12}\wLambda^{-1}\wY_{12}^\top - Y_{12}\Lambda^{-1}Y_{12}^\top - (\wY_{12} - \wY_{12})(\wLambda - \Lambda)^{-1} (\wY_{12} - Y_{12})^\top\\
        &= \wY_{12}\left(\wLambda^{-1} - (\wLambda-\Lambda)^{-1}\right)\wY_{12}^\top - Y_{12}\left(\Lambda^{-1} + (\wLambda-\Lambda)^{-1}\right)Y_{12}^\top + \wY_{12}(\wLambda-\Lambda)^{-1}Y_{12}^\top + Y_{12}(\wLambda-\Lambda)^{-1}\wY_{12}^\top.
    \end{align*}
    Since diagonal matrices are interchageable, let $M:=(\wLambda-\Lambda)^{-1}\wLambda^{-1}\Lambda^{-1} $, then
    \begin{align*}
        \wLambda^{-1} - (\wLambda-\Lambda)^{-1}&= M(\wLambda-\Lambda)\Lambda - M\wLambda\Lambda = -\Lambda M \Lambda\\
        \Lambda^{-1} + (\wLambda-\Lambda)^{-1} &= M(\wLambda-\Lambda)\wLambda + M\wLambda\Lambda = \wLambda M\wLambda\\
        (\wLambda-\Lambda)^{-1} &= \Lambda M \wLambda = \wLambda M \Lambda.
    \end{align*}
    Thus
    \begin{align*}
        \textup{Part A} &= -\wY_{12}\Lambda M \Lambda \wY_{12}^\top - Y_{12}\wLambda M\wLambda Y_{12}^\top + \wY_{12}\Lambda M \wLambda Y_{12} + Y_{12}\wLambda M\Lambda Y_{12}^\top\\
        &= - (\wY_{12}\Lambda-Y_{12}\wLambda) M (\wY_{12}\Lambda-Y_{12}\wLambda)^\top \preceq 0 .
    \end{align*}
    From \eqref{eq:schur-complement}, $\wDelta_{11} - \Delta_{11}\succeq -\textup{Part A} \succeq 0$, which completes the proof for the case where $\wX_{22} - X_{22}$ is invertible. For the case when  $\wX_{22} - X_{22}$ is not invertible, we consider a sequence of matrices $\wX^\epsilon$
    \begin{align*}
       \wX^\epsilon  = \left[\begin{array}{cc}
    \wX_{11} & \wX_{12} \\
    \wX_{12}^\top & \wX_{22} + \epsilon I
\end{array}\right].
    \end{align*}
    Since $\wX_{22} - X_22\succeq 0$, we have that $\wX_{22} + \epsilon I - X_{22}$ is invertible, thus
    \begin{equation*}
    \wX_{11} - \wX_{12}\left(\wX_{22} + \epsilon I\right)^{-1}\wX_{12}^\top\succeq X_{11} - X_{12}X_{22}^{-1}X_{12}^\top.
\end{equation*}
Since $\wX_{22}\succ 0$,  $$ \lim_{\epsilon\to 0}\wX_{11} - \wX_{12}\left(\wX_{22} + \epsilon I\right)^{-1}\wX_{12}^\top =  \wX_{11} - \wX_{12}\wX_{22}^{-1}\wX_{12}^\top,$$
thus letting $\epsilon\to 0$ completes the proof.
\end{proof}
\end{lemma}

\begin{lemma}\label{lemma:log-det-principal-submatrix}
Suppose $X\in \bR^{n\times n}$ is a positive-semi-definite matrix, and $X'\in \bR^{n'\times n'}$ where $n' < n$ is a principal submatrix of X. Then
\begin{equation*}
    \log(\det(I_{n'} + X')) \le \log(\det(I_{n}) + X).
\end{equation*}
\begin{proof}
    Assume that the eigenvalues of $X$ are
    \begin{align*}
        \lambda_1 \ge \lambda_2\ge \cdots\ge \lambda_n\ge 0,
    \end{align*}
    and the eigenvalues of $X'$ are
    \begin{align*}
        \lambda_1'  \ge \lambda_2'\ge \cdots\ge \lambda_{n'}'\ge 0.
    \end{align*}
    From Cauchy interlacing Theorem we have that
    \begin{align*}
        \lambda_1 \ge \lambda_1',~~ \lambda_2 \ge \lambda_2', ~~\cdots, ~~\lambda_{n'} \ge \lambda_{n'}'. 
    \end{align*}
    Thus
    \begin{align*}
        \log(\det(I_{n'} + X')) &= \sum_{i=1}^{n'}\log(1+\lambda_i') \le \sum_{i=1}^{n'}\log(1+\lambda_i)\\
        &\le \sum_{i=1}^{n}\log(1+\lambda_i) = \log(\det(I_n + X)),
    \end{align*}
    which completes the proof.
\end{proof}
\end{lemma}
\section{Detailed Experiment Settings and Additional Numerical Results}
\subsection{Complete Experimental Results}

The results of (i) observation noise variance $0.1,0.001$; (ii) another kernel length scale at $0.5$; (iii) more results on more agents are shown in Figure \ref{fig:apdx-numerical-main1}, \ref{fig:apdx-numerical-main2}, \ref{fig:apdx-numerical-main3}. The results have shown that the proposed MAC-DT algorithm stably achieved no-regret maximal coverage and outperformed two baselines.
\begin{figure}[ht]
    \centering
    \includegraphics[width=0.24\linewidth]{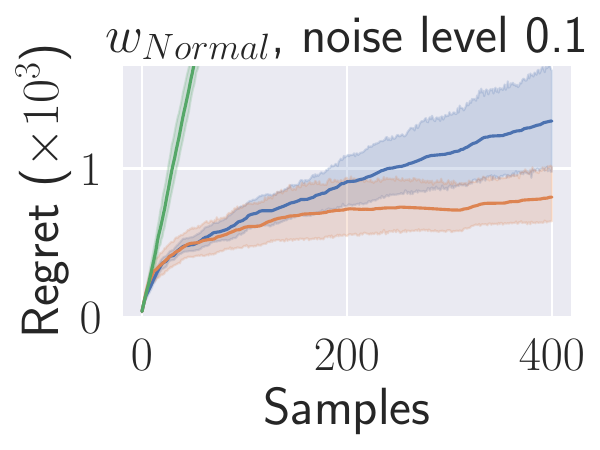}
    \includegraphics[width=0.24\linewidth]{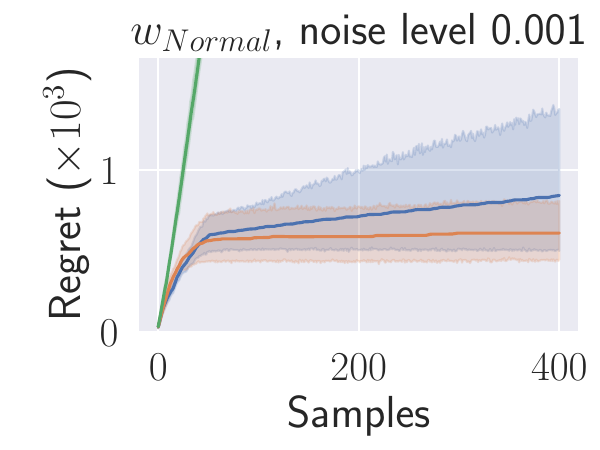}
    \includegraphics[width=0.24\linewidth]{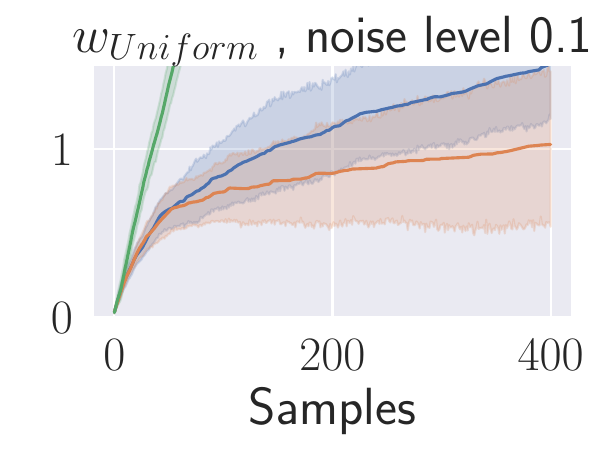}\\
    \includegraphics[width=0.24\linewidth]{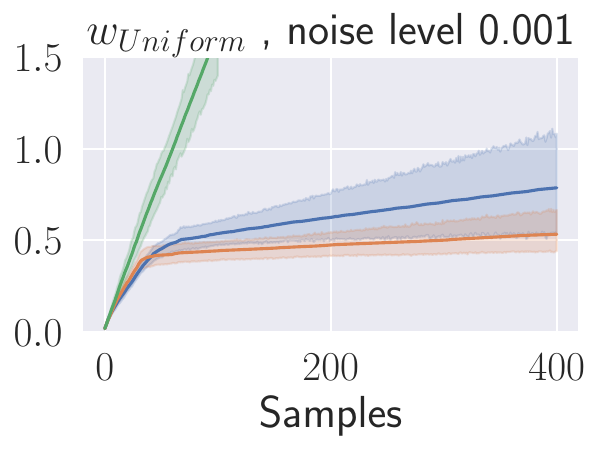}
    \includegraphics[width=0.24\linewidth]{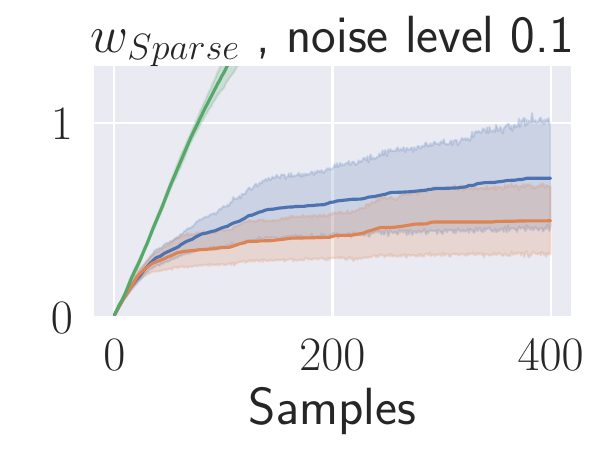}
    \includegraphics[width=0.24\linewidth]{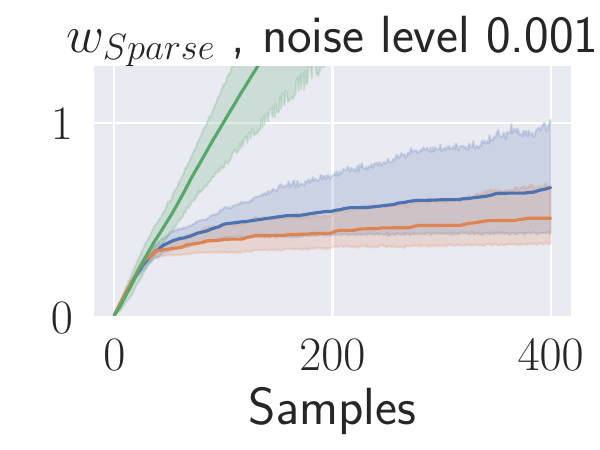}

    \includegraphics[width=0.6\linewidth]{exp_figures/legend.pdf}
    \caption{Complete results of kernel length scale 0.5 under different reward maps and observation noise levels. A kernel length scale of 0.5 means the GP assumes that neighborhood rewards are correlated.}
    \label{fig:apdx-numerical-main1}
\end{figure}
\begin{figure}[ht]
    \centering
    \includegraphics[width=0.24\linewidth]{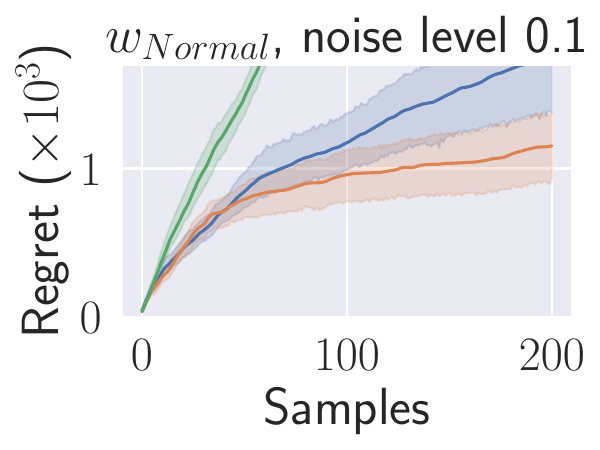}
    \includegraphics[width=0.24\linewidth]{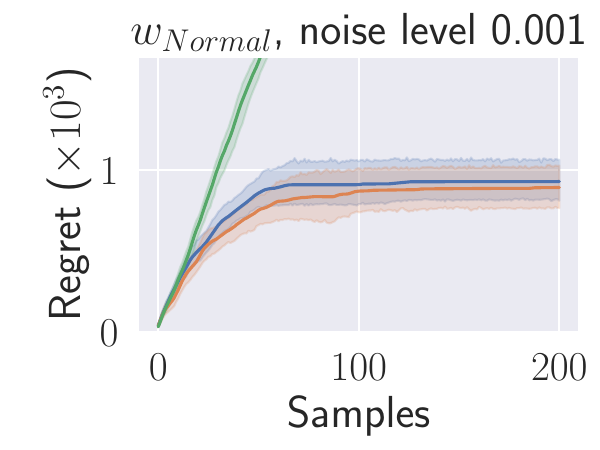}
    \includegraphics[width=0.24\linewidth]{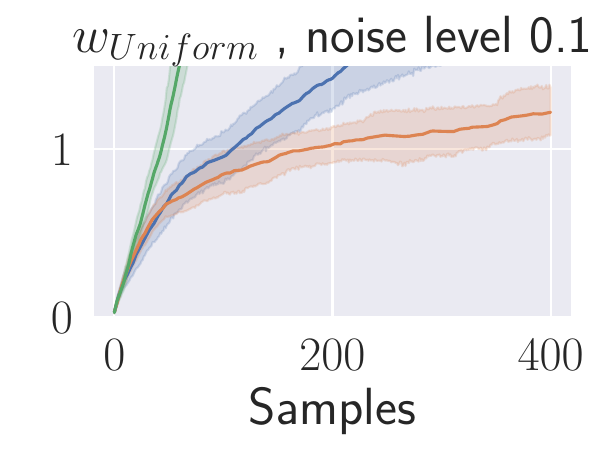}\\
    \includegraphics[width=0.24\linewidth]{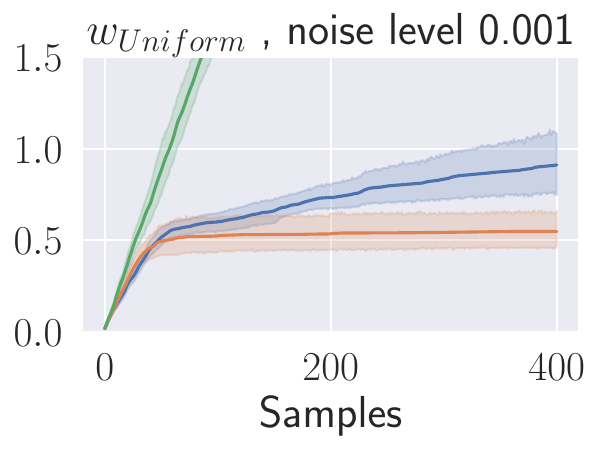}
    \includegraphics[width=0.24\linewidth]{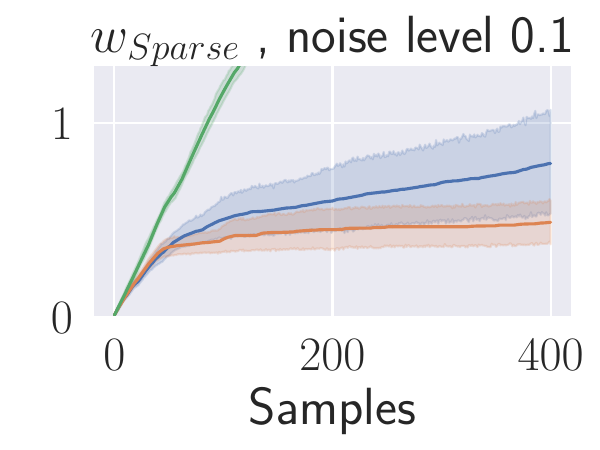}
    \includegraphics[width=0.24\linewidth]{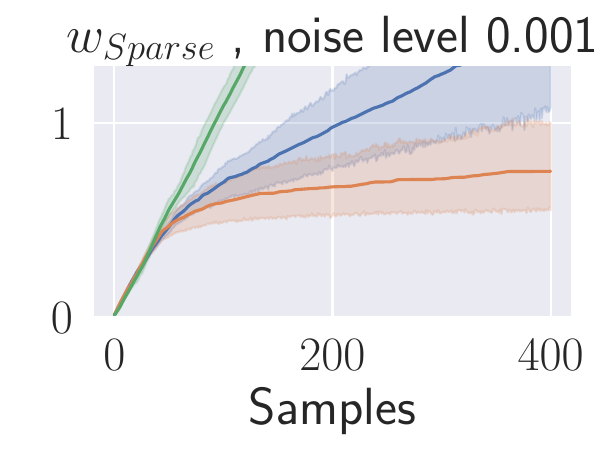}
    \includegraphics[width=0.6\linewidth]{exp_figures/legend.pdf}
    \caption{Complete results of kernel length scale $0.01$ under different reward maps and observation noise levels. A kernel length scale of 0.01 means the GP assumes that neighborhood rewards are approximately independent.}
    \label{fig:apdx-numerical-main2}
\end{figure}
\begin{figure}[ht]
    \centering
    \includegraphics[width=0.24\linewidth]{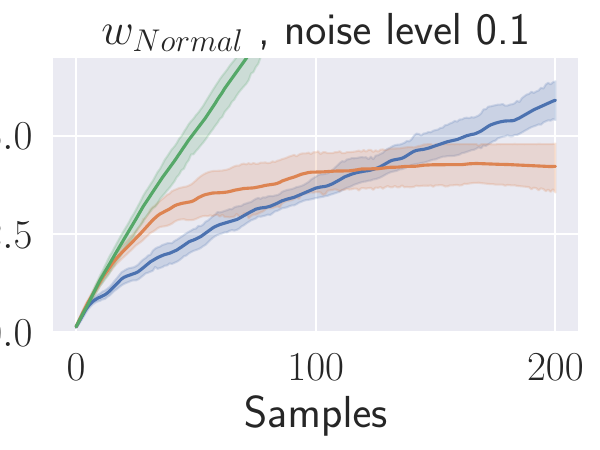}
    \includegraphics[width=0.24\linewidth]{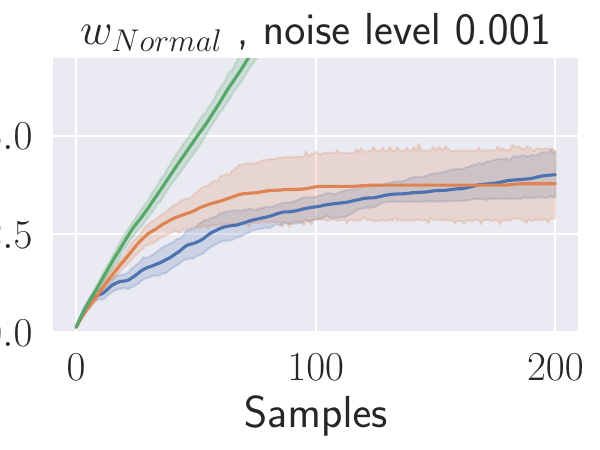}
    \includegraphics[width=0.24\linewidth]{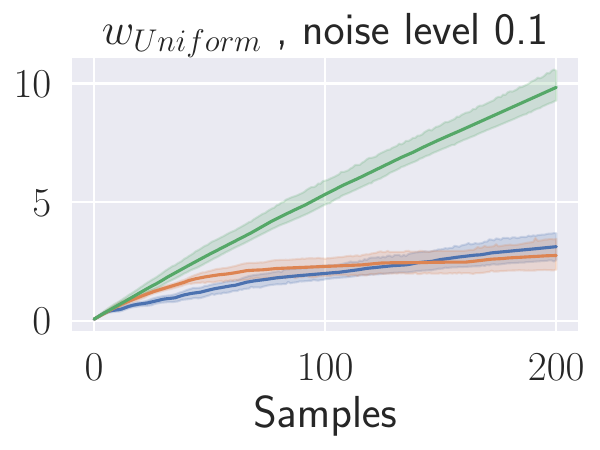}
    \includegraphics[width=0.24\linewidth]{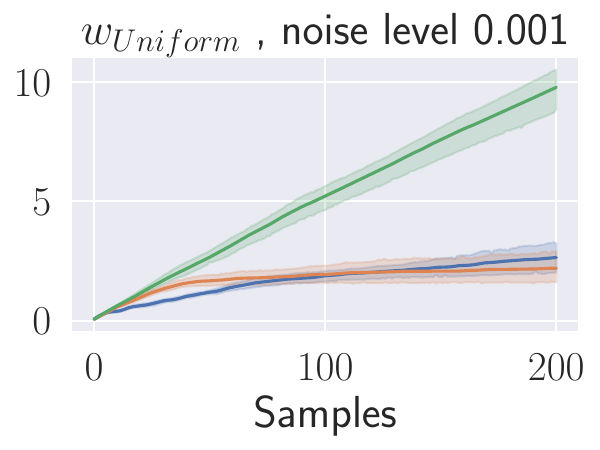}
    \includegraphics[width=0.6\linewidth]{exp_figures/legend.pdf}
    \caption{Complete results of 10 agents under different reward maps and observation noise levels.}
    \label{fig:apdx-numerical-main3}
\end{figure}
\subsection{Additional Experimental Details}
\label{sec:apdx_exp_texts}

\subsubsection{Voronoi Partition Coverage and Modifications}
The Voronoi partition algorithm is trying to find the best $N$ non-overlapped partition of a graph with node rewards, where the performance is defined by the coverage cost

\begin{equation}
    \mathcal{C}(\boldsymbol{\eta}, P)=\sum_{i=1}^N \sum_{v^{\prime} \in P_i} d_{G\left[P_i\right]}\left(\eta_i, v^{\prime}\right) w\left(v^{\prime}\right)\label{eq:voronoi}
\end{equation}
where $\boldsymbol{\eta}$ a list of $N$ robot configurations (robot positions in our case), $P_i\in V$ is is a subset of nodes indicating the Voronoi partition of $i^{th}$ robot, and $G[P_i]$ is a subgraph of $G$ that only includes nodes in $P_i$ and edges in $G$ which connects both nodes in $P_i$. $d$

The classic Lloyd algorithm iteratively places the robot to the centroid of the current Voronoi partition and computes the new Voronoi partition using the updated configuration. \citet{wei2021multi} integrate the Gaussian process to learn the node rewards. As we only consider coverage within the sensory area defined by the $k$-hop neighborhood, we define the distance function in Eq. \eqref{eq:voronoi} to be 
\begin{equation}
    d_{G[P_i]}(\eta_i, v)=\left\{\begin{matrix}
        1  & \text{if } v \text{ is in the k-hop neighborhood of } \eta_i
        \\0 & \text{otherwise}
    \end{matrix}\right.\label{eq:distance}
\end{equation}
. Therefore, the objective function in \eqref{eq:voronoi} defines the same multi-agent coverage problem as ours. Note that the exploration phase added in \citet{wei2021multi} is not very clear and does not consider the transient phase. To modify the exploration phase, we sequentially generate the target position for maximizing the uncertainty. At episode $e$, the target points for agent $i$ is selected by
\begin{equation}
    v_{i,dest} = \argmax_{v\in V}\hat \sigma_{i-1}^{(e)}(v)\label{eq:voronoi_sigma}
\end{equation}
where $\hat \sigma_{i-1}(v)  =\kappa(v,v) - \kappa(v, s) (\kappa(s,s) + \sigma^2 I)^{-1} \kappa(s,v)$, and $s$ include all query points in previous episodes and $v_{1, dest},v_{2, dest},\dots, v_{i-1, dest}$. The reason for the definition is that the update rules of covariance $\kappa'(u,v)$ in \eqref{eq:GP-update} are only related to the planned query points and not related to the observations. This is usually called imaginary queries, which is commonly used in multi-agent/batch Bayesian optimization. More discussion of the imaginary queries could be found in \citet{williams2006gaussian}. After we select one target, we can get the updated covariance and select the next target point for the next agents with maximum uncertainty. Therefore, we have our baseline algorithm in Algorithm \ref{algo:voronoi},

\begin{algorithm}[h]
    \caption{Multi-Agent Coverage by Voronoi Partition} \label{algo:voronoi}

  \begin{algorithmic}[1]
    \REQUIRE{GP model of the node rewards$ $}
    \FOR{episode $e = 1,2,...$}
    
    \STATE \texttt{\# Exploration phase}
    \STATE {Target selection: The robot team select target by $v_{i,dest}=\argmax_{v\in V}\sigma_i^{(e)}(v)$ according to \eqref{eq:voronoi_sigma}}.
    \STATE {Path planning and Collect samples: Agents plan paths to the targets following their planned paths. At time step $t$, each agent $i$ sets the evaluation point as $v_{i,t}^\eval = \argmax_{v\in s_{i,t}} \sigma^{(e-1)}(v)$, and collect sample $Y(v_{i,t}^\eval)$.} \label{algline:plan_and_collect_samples}
    \STATE Update the mean value $\bar\mu^{(e)}$ and $\bar\kappa^{(e)}$ covariance of GP according to \eqref{eq:GP-update}.
    \STATE \texttt{\# Coverage phase}
    \STATE Calculate the Voronoi partition according to \eqref{eq:voronoi} using the distance definition in \eqref{eq:distance} and posterior mean reward $\bar\mu^{(e)}$. Plan the path to the Voronoi partition destinations and collect samples like line \ref{algline:plan_and_collect_samples}.
    \STATE Given the current Voronoi partitions, calculate the center of mass, then plan the path to the center of mass and collect samples like line \ref{algline:plan_and_collect_samples}.
    \STATE Update the mean value $\mu^{(e)}$ and $\kappa^{(e)}$ covariance of GP for objective function according to \eqref{eq:GP-update}.
    \ENDFOR
  \end{algorithmic} 
    
\end{algorithm}

\textbf{Notations.} $\bar\mu^{(e)}$ and $\bar\kappa^{(e)}$ means intermediate posterior mean and covariance after the exploration phase at episode $(e)$.

        

\subsection{Safety Considerations}
\subsubsection{\texttt{Mac-DT-SafelyExplore} Algorithm details}
\label{apdx:safety_algo}
The \texttt{Mac-DT-SafelyExplore} algorithm is shown in Algorithm \ref{alg:safe}. Because we need to learn the safety function by maintaining another GP, we need to get observations from the environments 
\begin{equation}
    Y_g(v_{i,t}) = \{v\in s_{i,t}\mid g(v)+\epsilon\}\label{eq:safe_observation}
\end{equation}
. The observation model means we can collect observations of the safety function at every point of the coverage area. The setting is realistic since in real robotics applications, the safety objective is usually collision avoidance, and robots usually use sensors like cameras or lidars to measure the distance. These sensors can easily measure the safety objective function for a large region rather than a single point. We also maintain another GP $\GP_g(\mu_g^{(e)}, \kappa_g^{(e)})$ using data $D_g$, where $D_g$ includes all the history query points and noisy observations of the safety objective function (we will not define notations for data in $D_g$ for simplicity).

For the path planning in Line \ref{algline: pathplanningsafe} of Algorithm \ref{alg:safe}, there might be questions about how to guarantee that there always exists a solution to the path planning within the safe set. First, in \citet{prajapat2022near}, the planning guarantees that the targets always lies in the $\hat V^{(e)}_{\text{safe}}$. Then the safe expansion algorithm in \citet{prajapat2022near} guarantees that the $V^{(e-1)}_{\text{safe}}\subseteq V^{(e)}_{\text{safe}}$, which means the estimated safe set keep expanding. Therefore, the initial position $v_{i,t}$ must be in the estimated safe set. Therefore, both the start point and end point lie in the estimated safe set. We also assume that the safe set is connected \footnote{In numerical experiments, we provide prior knowledge about the safe set to guarantee the connectivity.}. Therefore, we can always find a path from the starting point to the goal.

\begin{algorithm}
    \caption{Mult-agent safe exploration with doubling trick (Mac-DT-SafelyExplore)} \label{alg:safe}
    \begin{algorithmic}[1]
        \FOR{episode $e = 1,2,...$}
        \STATE Given two GPs of objective function and safety function, set destinations according to the SafeMaC algorithm in \citet{prajapat2022near}. All hyperparameters follow the SafeMaC algorithm.
        \STATE Path Planning: Update the estimated safe region $\hat V^{(e)}_{\text{safe}}$. Compute the 
        weighted shortest path on the safe subgraph of $G$ from current location $v_{i,t_{e}}$ to $v_{i,dest}$. Safe subgraph of $G$ includes only nodes in $\hat V^{(e)}_{\text{safe}}$. Each edge entering node $v$ is weighted by the mean value of safety objective function $\mu_g(v)$. \label{algline: pathplanningsafe}
        \STATE Collect samples: Agents follow their planned paths and at time step $t$, each agent $i$ collects objective function sample $Y(v_{i,t}^\eval)$ by setting evaluation points $v_{i,t}^\eval = \argmax_{v\in s_{i,t}} \sigma^{(e-1)}(v)$ and safety function samples $Y_g(v_{i,t}^\eval)$ by \eqref{eq:safe_observation}. When an agent reaches its destination, it stays at the destination and keeps collecting samples until the episode terminates.
        \STATE Episode termination criteria \!-\! the doubling trick: episode terminates at a minimum $t$ such that there exists $v\in V$, such that $n_v(t) \ge \max\{2n_v(t_e-1),1\}$, i.e. when at least one of the vertex's objective function samples doubles.
        \STATE Update the mean value and covariance of both GPs for objective function and safety function according to \eqref{eq:GP-update}`.
    \ENDFOR 
    \end{algorithmic}
\end{algorithm}
\subsubsection{Additional analysis of the experimental results}
\label{apdx:safety_exp}
\texttt{Mac-DT-Safe} slightly outperforms \texttt{SafeMaC-SP} but not too much. The difference between these two are the termination criteria, \texttt{Mac-DT-Safe} has the doubling trick and \texttt{SafeMaC-SP} has the naive destination reached. h means adding the doubling trick slightly helps the performance but not too much. The reason is that regret is affected by both coverage optimality and safe exploration, whereas the doubling trick only improves optimality. The \texttt{SafePlanMac} achieves zero-regret coverage fast, showing that we can add simple rules to encourage safe exploration in the path planning to improve the performance. This shows the importance of considering the path planning problem during the transient phase.
\bibliography{bib}

\begin{thebibliography}{43}
\providecommand{\natexlab}[1]{#1}
\providecommand{\url}[1]{\texttt{#1}}
\expandafter\ifx\csname urlstyle\endcsname\relax
  \providecommand{\doi}[1]{doi: #1}\else
  \providecommand{\doi}{doi: \begingroup \urlstyle{rm}\Url}\fi

\bibitem[Auer et~al.(2008)Auer, Jaksch, and Ortner]{auer2008near}
Peter Auer, Thomas Jaksch, and Ronald Ortner.
\newblock Near-optimal regret bounds for reinforcement learning.
\newblock \emph{Advances in neural information processing systems}, 21, 2008.

\bibitem[Battocletti et~al.(2021)Battocletti, Urban, Godio, and
  Guglieri]{battocletti2021rl}
Gianpietro Battocletti, Riccardo Urban, Simone Godio, and Giorgio Guglieri.
\newblock Rl-based path planning for autonomous aerial vehicles in unknown
  environments.
\newblock In \emph{AIAA AVIATION 2021 FORUM}, page 3016, 2021.

\bibitem[Benevento et~al.(2020)Benevento, Santos, Notarstefano, Paynabar,
  Bloch, and Egerstedt]{benevento2020multi}
Alessia Benevento, Mar{\'\i}a Santos, Giuseppe Notarstefano, Kamran Paynabar,
  Matthieu Bloch, and Magnus Egerstedt.
\newblock Multi-robot coordination for estimation and coverage of unknown
  spatial fields.
\newblock In \emph{2020 ieee international conference on robotics and
  automation (icra)}, pages 7740--7746. IEEE, 2020.

\bibitem[Bullo et~al.(2012)Bullo, Carli, and Frasca]{bullo2012gossip}
Francesco Bullo, Ruggero Carli, and Paolo Frasca.
\newblock Gossip coverage control for robotic networks: Dynamical systems on
  the space of partitions.
\newblock \emph{SIAM Journal on Control and Optimization}, 50\penalty0
  (1):\penalty0 419--447, 2012.

\bibitem[Carron et~al.(2015)Carron, Todescato, Carli, Schenato, and
  Pillonetto]{carron2015multi}
Andrea Carron, Marco Todescato, Ruggero Carli, Luca Schenato, and Gianluigi
  Pillonetto.
\newblock Multi-agents adaptive estimation and coverage control using gaussian
  regression.
\newblock In \emph{2015 European Control Conference (ECC)}, pages 2490--2495.
  IEEE, 2015.

\bibitem[Choi and Horowitz(2010)]{choi2010learning}
Jongeun Choi and Roberto Horowitz.
\newblock Learning coverage control of mobile sensing agents in one-dimensional
  stochastic environments.
\newblock \emph{IEEE Transactions on Automatic Control}, 55\penalty0
  (3):\penalty0 804--809, 2010.

\bibitem[Choi et~al.(2008)Choi, Lee, and Oh]{choi2008swarm}
Jongeun Choi, Joonho Lee, and Songhwai Oh.
\newblock Swarm intelligence for achieving the global maximum using
  spatio-temporal gaussian processes.
\newblock In \emph{2008 American Control Conference}, pages 135--140. IEEE,
  2008.

\bibitem[Cort{\'e}s and Bullo(2005)]{cortes2005coordination}
Jorge Cort{\'e}s and Francesco Bullo.
\newblock Coordination and geometric optimization via distributed dynamical
  systems.
\newblock \emph{SIAM journal on control and optimization}, 44\penalty0
  (5):\penalty0 1543--1574, 2005.

\bibitem[Cortes et~al.(2004)Cortes, Martinez, Karatas, and
  Bullo]{cortes2004coverage}
Jorge Cortes, Sonia Martinez, Timur Karatas, and Francesco Bullo.
\newblock Coverage control for mobile sensing networks.
\newblock \emph{IEEE Transactions on robotics and Automation}, 20\penalty0
  (2):\penalty0 243--255, 2004.

\bibitem[Cortes et~al.(2005)Cortes, Martinez, and Bullo]{cortes2005spatially}
Jorge Cortes, Sonia Martinez, and Francesco Bullo.
\newblock Spatially-distributed coverage optimization and control with
  limited-range interactions.
\newblock \emph{ESAIM: Control, Optimisation and Calculus of Variations},
  11\penalty0 (4):\penalty0 691--719, 2005.

\bibitem[Davison et~al.(2014)Davison, Leonard, Olshevsky, and
  Schwemmer]{davison2014nonuniform}
Peter Davison, Naomi~Ehrich Leonard, Alex Olshevsky, and Michael Schwemmer.
\newblock Nonuniform line coverage from noisy scalar measurements.
\newblock \emph{IEEE Transactions on Automatic Control}, 60\penalty0
  (7):\penalty0 1975--1980, 2014.

\bibitem[Din et~al.(2022)Din, Ismail, Shah, Babar, Ali, and Baig]{DIN2022deep}
Ahmad Din, Muhammed~Yousoof Ismail, Babar Shah, Mohammad Babar, Farman Ali, and
  Siddique~Ullah Baig.
\newblock A deep reinforcement learning-based multi-agent area coverage control
  for smart agriculture.
\newblock \emph{Computers and Electrical Engineering}, 101:\penalty0 108089,
  2022.
\newblock ISSN 0045-7906.
\newblock \doi{https://doi.org/10.1016/j.compeleceng.2022.108089}.
\newblock URL
  \url{https://www.sciencedirect.com/science/article/pii/S0045790622003445}.

\bibitem[Durham et~al.(2011)Durham, Carli, Frasca, and
  Bullo]{durham2011discrete}
Joseph~W Durham, Ruggero Carli, Paolo Frasca, and Francesco Bullo.
\newblock Discrete partitioning and coverage control for gossiping robots.
\newblock \emph{IEEE Transactions on Robotics}, 28\penalty0 (2):\penalty0
  364--378, 2011.

\bibitem[Faryadi and Mohammadpour~Velni(2021)]{faryadi2021reinforcement}
Saba Faryadi and Javad Mohammadpour~Velni.
\newblock A reinforcement learning-based approach for modeling and coverage of
  an unknown field using a team of autonomous ground vehicles.
\newblock \emph{International journal of intelligent systems}, 36\penalty0
  (2):\penalty0 1069--1084, 2021.

\bibitem[Feige(1998)]{feige1998threshold}
Uriel Feige.
\newblock A threshold of ln n for approximating set cover.
\newblock \emph{Journal of the ACM (JACM)}, 45\penalty0 (4):\penalty0 634--652,
  1998.

\bibitem[Hussein and Stipanovic(2007)]{hussein2007effective}
Islam~I Hussein and Dusan~M Stipanovic.
\newblock Effective coverage control for mobile sensor networks with guaranteed
  collision avoidance.
\newblock \emph{IEEE Transactions on Control Systems Technology}, 15\penalty0
  (4):\penalty0 642--657, 2007.

\bibitem[Karapetyan et~al.(2018)Karapetyan, Moulton, Lewis, Li, O'Kane, and
  Rekleitis]{karapetyan2018multi}
Nare Karapetyan, Jason Moulton, Jeremy~S Lewis, Alberto~Quattrini Li, Jason~M
  O'Kane, and Ioannis Rekleitis.
\newblock Multi-robot dubins coverage with autonomous surface vehicles.
\newblock In \emph{2018 IEEE International Conference on Robotics and
  Automation (ICRA)}, pages 2373--2379. IEEE, 2018.

\bibitem[Kemna et~al.(2017)Kemna, Rogers, Nieto-Granda, Young, and
  Sukhatme]{kemna2017multi}
Stephanie Kemna, John~G Rogers, Carlos Nieto-Granda, Stuart Young, and Gaurav~S
  Sukhatme.
\newblock Multi-robot coordination through dynamic voronoi partitioning for
  informative adaptive sampling in communication-constrained environments.
\newblock In \emph{2017 IEEE International Conference on Robotics and
  Automation (ICRA)}, pages 2124--2130. IEEE, 2017.

\bibitem[Krause et~al.(2006)Krause, Guestrin, Gupta, and Kleinberg]{Krause06}
A.~Krause, C.~Guestrin, A.~Gupta, and J.~Kleinberg.
\newblock Near-optimal sensor placements: maximizing information while
  minimizing communication cost.
\newblock In \emph{2006 5th International Conference on Information Processing
  in Sensor Networks}, pages 2--10, 2006.
\newblock \doi{10.1145/1127777.1127782}.

\bibitem[Krause and Golovin(2014)]{krause2014submodular}
Andreas Krause and Daniel Golovin.
\newblock Submodular function maximization.
\newblock \emph{Tractability}, 3\penalty0 (71-104):\penalty0 3, 2014.

\bibitem[Krause and Guestrin(2011)]{krause2011submodularity}
Andreas Krause and Carlos Guestrin.
\newblock Submodularity and its applications in optimized information
  gathering.
\newblock \emph{ACM Transactions on Intelligent Systems and Technology (TIST)},
  2\penalty0 (4):\penalty0 1--20, 2011.

\bibitem[Lekien and Leonard(2009)]{lekien2009nonuniform}
Francois Lekien and Naomi~Ehrich Leonard.
\newblock Nonuniform coverage and cartograms.
\newblock \emph{SIAM Journal on Control and Optimization}, 48\penalty0
  (1):\penalty0 351--372, 2009.

\bibitem[Lloyd(1982)]{lloyd1982least}
Stuart Lloyd.
\newblock Least squares quantization in pcm.
\newblock \emph{IEEE transactions on information theory}, 28\penalty0
  (2):\penalty0 129--137, 1982.

\bibitem[Luo and Sycara(2018)]{luo2018adaptive}
Wenhao Luo and Katia Sycara.
\newblock Adaptive sampling and online learning in multi-robot sensor coverage
  with mixture of gaussian processes.
\newblock In \emph{2018 IEEE international conference on robotics and
  automation (ICRA)}, pages 6359--6364. IEEE, 2018.

\bibitem[Luo et~al.(2019)Luo, Nam, Kantor, and Sycara]{luo2019distributed}
Wenhao Luo, Changjoo Nam, George Kantor, and Katia Sycara.
\newblock Distributed environmental modeling and adaptive sampling for
  multi-robot sensor coverage.
\newblock In \emph{Proceedings of the 18th International Conference on
  Autonomous Agents and MultiAgent Systems}, pages 1488--1496, 2019.

\bibitem[MacKay(2003)]{mackay2003information}
David~JC MacKay.
\newblock \emph{Information theory, inference and learning algorithms}.
\newblock Cambridge university press, 2003.

\bibitem[Mainwaring et~al.(2002)Mainwaring, Culler, Polastre, Szewczyk, and
  Anderson]{mainwaring2002wireless}
Alan Mainwaring, David Culler, Joseph Polastre, Robert Szewczyk, and John
  Anderson.
\newblock Wireless sensor networks for habitat monitoring.
\newblock In \emph{Proceedings of the 1st ACM international workshop on
  Wireless sensor networks and applications}, pages 88--97, 2002.

\bibitem[Nakamura et~al.(2022)Nakamura, Santos, and
  Leonard]{nakamura2022decentralized}
Kensuke Nakamura, Mar{\'\i}a Santos, and Naomi~Ehrich Leonard.
\newblock Decentralized learning with limited communications for multi-robot
  coverage of unknown spatial fields.
\newblock In \emph{2022 IEEE/RSJ International Conference on Intelligent Robots
  and Systems (IROS)}, pages 9980--9986. IEEE, 2022.

\bibitem[Nemhauser et~al.(1978)Nemhauser, Wolsey, and
  Fisher]{nemhauser1978analysis}
George~L Nemhauser, Laurence~A Wolsey, and Marshall~L Fisher.
\newblock An analysis of approximations for maximizing submodular set
  functions—i.
\newblock \emph{Mathematical programming}, 14:\penalty0 265--294, 1978.

\bibitem[Prajapat et~al.(2022)Prajapat, Turchetta, Zeilinger, and
  Krause]{prajapat2022near}
Manish Prajapat, Matteo Turchetta, Melanie Zeilinger, and Andreas Krause.
\newblock Near-optimal multi-agent learning for safe coverage control.
\newblock \emph{Advances in Neural Information Processing Systems},
  35:\penalty0 14998--15012, 2022.

\bibitem[Ramaswamy and Marden(2016)]{ramaswamy2016sensor}
Vinod Ramaswamy and Jason~R Marden.
\newblock A sensor coverage game with improved efficiency guarantees.
\newblock In \emph{2016 American Control Conference (ACC)}, pages 6399--6404.
  IEEE, 2016.

\bibitem[Santos et~al.(2021)Santos, Madhushani, Benevento, and
  Leonard]{santos2021multi}
Maria Santos, Udari Madhushani, Alessia Benevento, and Naomi~Ehrich Leonard.
\newblock Multi-robot learning and coverage of unknown spatial fields.
\newblock In \emph{2021 International Symposium on Multi-Robot and Multi-Agent
  Systems (MRS)}, pages 137--145. IEEE, 2021.

\bibitem[Schwager et~al.(2009)Schwager, Rus, and
  Slotine]{schwager2009decentralized}
Mac Schwager, Daniela Rus, and Jean-Jacques Slotine.
\newblock Decentralized, adaptive coverage control for networked robots.
\newblock \emph{The International Journal of Robotics Research}, 28\penalty0
  (3):\penalty0 357--375, 2009.

\bibitem[Schwager et~al.(2015)Schwager, Vitus, Powers, Rus, and
  Tomlin]{schwager2015robust}
Mac Schwager, Michael~P Vitus, Samantha Powers, Daniela Rus, and Claire~J
  Tomlin.
\newblock Robust adaptive coverage control for robotic sensor networks.
\newblock \emph{IEEE Transactions on Control of Network Systems}, 4\penalty0
  (3):\penalty0 462--476, 2015.

\bibitem[Srinivas et~al.(2009)Srinivas, Krause, Kakade, and
  Seeger]{srinivas2009gaussian}
Niranjan Srinivas, Andreas Krause, Sham~M Kakade, and Matthias Seeger.
\newblock Gaussian process optimization in the bandit setting: No regret and
  experimental design.
\newblock \emph{arXiv preprint arXiv:0912.3995}, 2009.

\bibitem[Sui et~al.(2015)Sui, Gotovos, Burdick, and Krause]{sui2015safe}
Yanan Sui, Alkis Gotovos, Joel Burdick, and Andreas Krause.
\newblock Safe exploration for optimization with gaussian processes.
\newblock In \emph{International conference on machine learning}, pages
  997--1005. PMLR, 2015.

\bibitem[Sun et~al.(2017)Sun, Cassandras, and Meng]{sun2017submodularity}
Xinmiao Sun, Christos~G Cassandras, and Xiangyu Meng.
\newblock A submodularity-based approach for multi-agent optimal coverage
  problems.
\newblock In \emph{2017 IEEE 56th Annual Conference on Decision and Control
  (CDC)}, pages 4082--4087. IEEE, 2017.

\bibitem[Todescato et~al.(2017)Todescato, Carron, Carli, Pillonetto, and
  Schenato]{todescato2017multi}
Marco Todescato, Andrea Carron, Ruggero Carli, Gianluigi Pillonetto, and Luca
  Schenato.
\newblock Multi-robots gaussian estimation and coverage control: From
  client--server to peer-to-peer architectures.
\newblock \emph{Automatica}, 80:\penalty0 284--294, 2017.

\bibitem[Wasim et~al.(2020)Wasim, Kashino, Nejat, and Benhabib]{Wasim}
Shiraz Wasim, Zendai Kashino, Goldie Nejat, and Beno Benhabib.
\newblock Directional-sensor network deployment planning for mobile-target
  search.
\newblock \emph{Robotics}, 9\penalty0 (4), 2020.
\newblock ISSN 2218-6581.
\newblock \doi{10.3390/robotics9040082}.
\newblock URL \url{https://www.mdpi.com/2218-6581/9/4/82}.

\bibitem[Wei et~al.(2021)Wei, McDonald, and Srivastava]{wei2021multi}
Lai Wei, Andrew McDonald, and Vaibhav Srivastava.
\newblock Multi-robot gaussian process estimation and coverage: Deterministic
  sequencing algorithm and regret analysis.
\newblock In \emph{2021 IEEE International Conference on Robotics and
  Automation (ICRA)}, pages 9080--9085. IEEE, 2021.

\bibitem[Williams and Rasmussen(2006)]{williams2006gaussian}
Christopher~KI Williams and Carl~Edward Rasmussen.
\newblock \emph{Gaussian processes for machine learning}, volume~2.
\newblock MIT press Cambridge, MA, 2006.

\bibitem[Zhang et~al.(2023{\natexlab{a}})Zhang, Ma, and Li]{supp}
Runyu Zhang, Haitong Ma, and Na~Li, 2023{\natexlab{a}}.
\newblock URL
  \url{{https://drive.google.com/file/d/1FIMZjhENcXvGKrWAhNFvQByA8ryT7yOn/view}}.

\bibitem[Zhang et~al.(2023{\natexlab{b}})Zhang, Johansson, and
  Li]{zhang2023multi}
Tianpeng Zhang, Kasper Johansson, and Na~Li.
\newblock Multi-armed bandit learning on a graph.
\newblock In \emph{2023 57th Annual Conference on Information Sciences and
  Systems (CISS)}, pages 1--6. IEEE, 2023{\natexlab{b}}.

\end{thebibliography}

\end{document}